\documentclass[11pt]{article}

\usepackage{fullpage}
\usepackage{amsmath,amssymb,amsthm,mathtools}
\usepackage{enumitem}
\usepackage{microtype}
\usepackage[authoryear,round]{natbib}
\usepackage[
  colorlinks=true,
  linkcolor=blue,
  citecolor=blue,
  urlcolor=blue
]{hyperref}
\usepackage{bm}
\usepackage{bbm}
\usepackage{xcolor}
\usepackage{comment}
\usepackage{tikz}
\usepackage{tkz-graph}

\DeclareFontFamily{U}{mathx}{}
\DeclareFontShape{U}{mathx}{m}{n}{<-> mathx10}{}
\DeclareSymbolFont{mathx}{U}{mathx}{m}{n}
\DeclareMathAccent{\widehat}{0}{mathx}{"70}
\DeclareMathAccent{\widecheck}{0}{mathx}{"71}

\allowdisplaybreaks
\numberwithin{equation}{section}

\newtheorem{theorem}{Theorem}[section]
\newtheorem{proposition}[theorem]{Proposition}
\newtheorem{lemma}[theorem]{Lemma}
\newtheorem{corollary}[theorem]{Corollary}
\theoremstyle{remark}

\newcommand{\R}{\mathbb R}
\newcommand{\E}{\mathbb E}
\newcommand{\Pp}{\mathbb P}
\newcommand{\1}{\mathbf 1}
\newcommand{\op}{\mathrm{op}}
\newcommand{\BL}{\mathrm{BL}}

\newcommand{\dd}{\,\mathrm d}
\newcommand{\e}{\mathrm e}

\newcommand{\cE}{\mathcal{E}}
\newcommand{\cF}{\mathcal{F}}
\newcommand{\cG}{\mathcal{G}}

\newcommand{\cL}{\mathcal{L}}

\newcommand{\cN}{\mathcal{N}}

\newcommand{\EE}{\mathbb{E}}

\newcommand{\SSS}{\mathbb{S}}

\makeatletter
\DeclareRobustCommand\Equiv{\mathrel{
\mathchoice
{\Equiv@\textfont\displaystyle{.45}}
{\Equiv@\textfont\textstyle{.45}}
{\Equiv@\scriptfont\scriptstyle{.5}}
{\Equiv@\scriptscriptfont\scriptscriptstyle{.55}}
}}
\newcommand{\Equiv@}[3]{
\rlap{\raisebox{#3\fontdimen5#12}{$\m@th#2 = $}}
\raisebox{-#3\fontdimen5#12}{$\m@th#2 = $}
}
\makeatother

\newcommand{\TV}{\operatorname{\mathsf{TV}}}

\newcommand{\Unif}{\operatorname{\mathsf{Unif}}}

\newcommand{\tr}{\operatorname{\mathsf{Tr}}}

\newcommand{\ud}{\mathrm{d}}

\newcommand{\Cov}{\operatorname{\mathsf{Cov}}}

\newcommand{\Var}{\operatorname{\mathsf{Var}}}

\newcommand{\convd}{\overset{\mathrm d}{\longrightarrow}}

\newcommand{\bitm}{\begin{itemize}[leftmargin=*]}
\newcommand{\eitm}{\end{itemize}}

\newcommand{\benm}{\begin{enumerate}[leftmargin=*]}
\newcommand{\eenm}{\end{enumerate}}

\definecolor{forestgreen}{rgb}{0.13, 0.55, 0.13}

\def\abs#1{\left| #1 \right|}
\newcommand{\norm}[1]{\left\lVert#1\right\rVert}

\newcommand{\inparen}[1]{\left(#1\right)}             
\newcommand{\inbraces}[1]{\left\{#1\right\}}           
\newcommand{\insquare}[1]{\left[#1\right]}

\newenvironment{keywords}{\par\noindent\textbf{Keywords:} }{\par}

\title{Optimal detection of geometry in random graphs and kernel matrices at critical thresholds}
\author{Cheng Mao\thanks{School of Mathematics, Georgia Institute
of Technology, cheng.mao@math.gatech.edu} \and
Timothy L.\ H.\ Wee\thanks{School of Mathematics, Georgia Institute
of Technology, timothy.wee@gatech.edu}}
\date{}
\hypersetup{
  pdftitle={Optimal detection of geometry in random graphs and kernel matrices at critical thresholds},
  pdfauthor={Cheng Mao and Timothy L. H. Wee}
}

\begin{document}

\maketitle

\begin{abstract}
Detection of geometry in random graphs and matrices has been extensively studied in recent years. For testing a dense random geometric graph on $n$ vertices with latent points in dimension $d$ against an Erd\H{o}s--R\'enyi graph, the detection threshold $d\asymp n^3$ is well known since Bubeck, Ding, Eldan, and R\'acz (2016). We sharpen this result by establishing the optimal asymptotic testing error in the critical regime $d/n^3\to\lambda\in(0,\infty)$. The log-likelihood ratio is shown to be asymptotically normal and dominated by the signed triangle count. Moreover, we consider testing the rank statistics of inner products between latent points against uniformly random ranks. We show that ranks preserve all asymptotic testing information available in the full matrix models, and a cubic statistic based on rank statistics attains the optimal asymptotic error. This then yields an asymptotically optimal test for random kernel matrices with a possibly unknown, strictly monotone kernel. The proofs combine the classical framework of Davies (1985) and Le Cam and Yang (1988) for asymptotic normality under information loss with conditional limit theorems for the models in consideration.
\end{abstract}


\tableofcontents

\section{Introduction}

\subsection{Detection of geometry in random graphs and matrices}
Since the work of \citet{bubeck2016testing}, detection of high-dimensional geometry in random graphs has received considerable 
attention in various settings
\citep{brennan2020phase,liu2022testing,liuracz2023probabilistic,bangachev2024fourier,mao2026random,gobel2026information,du2026resolution,hunter2026distinguishability}.
More specifically, let $x_1,\ldots,x_n$ be independent uniformly random points on the unit sphere in
$\R^d$, with both $n$ and $d$ growing. A random geometric graph connects
vertices $i$ and $j$ when $\langle x_i,x_j\rangle$ exceeds a threshold chosen
to give a fixed edge density $p\in(0,1)$. \citet{bubeck2016testing} showed
that the signed triangle count consistently distinguishes this graph from
an Erd\H{o}s--R\'enyi graph $G(n,p)$ if $d\ll n^3$. If $d\gg n^3$,
the total variation (TV) distance between the two models tends to zero, making
detection impossible.

The same scale $d \asymp n^3$ governs the transition from the Wishart distribution to the Gaussian orthogonal ensemble (GOE) of random matrices \citep{jiangli2015approximation,bubeck2016testing}. 
See also \citet{bubeck2018entropic} and \citet{brennan2021finetti} for extensions to more general settings. 
One explanation of why the detection thresholds for dense random graphs and matrices share the same scale $d \asymp n^3$ is that graphs can be obtained from matrices by thresholding matrix entries, as shown by \citet{bubeck2016testing} on the indistinguishability side.

In addition, a finer result is known for random matrices: \citet{racz2019smooth} determined the exact asymptotic TV
distance between Wishart and GOE in the critical regime $d/n^3\to\lambda\in(0,\infty)$ by analyzing
the explicit likelihood ratio. However, such a sharp result at criticality was not known for detecting geometry in random graphs. 
Recently, \citet{hunter2026distinguishability} obtained an upper bound on the log-likelihood ratio whose leading term is the signed triangle count, but they did not pursue the exact asymptotics in the critical regime considered here.

In this work, we study the transition from the random geometric graph model to the Erd\H{o}s--R\'enyi model in the critical regime $d/n^3\to\lambda$. The log-likelihood ratio between the two graph models is shown to be asymptotically normal, with its
leading random term given by the signed triangle count. This yields the
exact asymptotic TV distance and establishes the asymptotic
optimality of the signed triangle count as a test statistic throughout the critical regime. 
In particular, to determine the transition between the two models, we show that the loss of information in the full graph due to thresholding matrix entries is entirely determined by an entrywise correlation parameter.

\subsection{Kernel matrices and rank statistics}
Moreover, we consider detection of geometry in random inner-product kernel matrices using the rank statistics of their entries. The spectra of such matrices have
been studied extensively \citep{elkaroui2010spectrum,cheng2013innerproduct,fan2019spectral}, and existing work also concerns recovery of latent geometry from kernel matrices \citep{whiteley2021matrix,giraud2023localization}.
Here we consider the more primitive question of detecting the latent geometry: Suppose we observe the off-diagonal entries of $\kappa(Y)$, where $\kappa$ is a strictly 
increasing (or decreasing) kernel applied entrywise to a matrix $Y$. We test whether
$Y$ is from the Wishart distribution or the GOE.\footnote{We
emphasize that $\kappa(Y)$ is observed directly here. This differs from detecting geometry
in a graph whose edges are sampled with probabilities given by a kernel \citep{liuracz2023probabilistic,mao2026random}.}
This testing problem would be the same as Wishart versus GOE if the kernel $\kappa$ were known because the strictly monotone kernel could be inverted. Without assuming the knowledge of $\kappa$, we provide a computationally efficient test that attains the optimal asymptotic
error in the critical regime $d/n^3\to\lambda$.

The test relies on a simple observation: a strictly increasing (resp.~decreasing) kernel $\kappa$ 
preserves (resp.~reverses) the global ordering (i.e., ranks) of the entries of $Y$. Although rank statistics discard cardinal values, 
they can 
preserve the information needed for asymptotically optimal inference in suitable models 
\citep{hajek1999theory,hallin1988optimal,bickel1997local,hoff2014information}.
Our result establishes this phenomenon for detection of geometry in random matrices. In particular, the asymptotically optimal statistic is a cubic polynomial of Gaussian quantiles defined using the ranks, analogous to the signed triangle count in the graph case. 

We note a few more related lines of research. First, the idea of using ranks of inner products or distances has a long history rooted in multidimensional scaling and ordinal embedding \citep{shepard1962analysis,kruskal1964multidimensional,kleindessner2014uniqueness}. Particularly, a recent independent work of \citet{coda2026theoretical} studies spherical ordinal embedding using rank statistics, which concerns recovery of geometry rather than detection. Moreover, a closely related work of \citet{liao2026testing} (see also an earlier work of \citet{ariascastro2017distribution}) studies detection of latent structure using rank statistics, but their alternative model is different from ours and concerns community detection. Finally, statistics of the global ranks of matrix entries have also been studied in biological settings for neural correlations \citep{giusti2015clique} and in other real-world scenarios \citep{caputi2026integral}, showing the wide use of rank statistics beyond theory.

\subsection{Asymptotic normality under information loss}
Our proofs for random graphs and rank statistics are unified under the classical framework of \citet{davies1985asymptotic} and \citet{lecam1988preservation}
for asymptotic normality under information loss. 
For an overview of the proof, let us consider random graphs obtained from thresholding Wishart and GOE matrices. 
The starting point of the analysis is a cubic expansion of the log-likelihood ratio between the parent models, Wishart and GOE. 
Then we split its leading term into an observable component (a multiple of the signed triangle count) and a
hidden residual, and obtain the likelihood ratio between random graph models by conditioning the parent likelihood ratio. 
The result of \citet{davies1985asymptotic} then 
shows that, if the observable component is asymptotically normal and the
residual has a conditional normal limit, then the graph log-likelihood
ratio is also asymptotically normal, with information given by the
limiting variance of the observable component
(Theorem~\ref{thm:coarsening}). 
Our main technical work lies in proving the needed conditional limit theorems for the models in consideration.
For the graph models, the information loss is characterized precisely by a correlation parameter, while for rank statistics, the residual vanishes, showing that the full information is preserved.

\paragraph{Notation.}
Write \([n]=\{1,\ldots,n\}\) and let \(\binom{[n]}2\) denote the set
of unordered pairs of distinct elements of \([n]\).
Let \(\Phi\) and \(\phi\) denote the standard Gaussian cumulative distribution function and probability density function respectively.
In the notation \(\cN(\mu,v)\) for Gaussian with mean \(\mu\) and variance
\(v \ge 0\), we allow \(v=0\) to mean the point mass at \(\mu\).
The notation \(\convd\) denotes convergence in distribution.
For a sequence of random variables \((X_n)_{n=1}^\infty\), \(X_n=o_{Q_n}(1)\) means convergence to zero
in \(Q_n\)-probability, while \(X_n=O_{Q_n}(1)\) means boundedness in
\(Q_n\)-probability.
Write \(\TV(P,Q)\) for the total variation distance between probability
distributions $P$ and $Q$.

\section{Main results}\label{sec:main-results}

Throughout, let $n$ denote the number of vertices in a graph or the size of a matrix. All the limits are taken as $n \to \infty$ unless otherwise specified. Let $d = d_n$ be a dimension parameter that can grow with $n$. Let \(x_1,\ldots,x_n\) be i.i.d.\ uniform points on the unit sphere \(\SSS^{d-1}\subset\R^d\). 

Before stating our main results, we set the stage by showing a spherical version of the result of \cite{racz2019smooth} on the comparison between Wishart and GOE. Suppose that, under \(P_{n,d}\), we observe the
symmetric matrix \(Y\) with zero diagonal and entries
\begin{equation*}
Y_{ij}=\sqrt d\langle x_i,x_j\rangle,\qquad 1\le i<j\le n,
\end{equation*}
where $x_1, \dots, x_n$ are the \emph{latent} points defined above. 
Under \(Q_n\), the matrix \(Y\) has 
zero diagonal and entries \((Y_{ij})_{i<j}\) that are independent standard
Gaussians. The asymptotic results concern the critical regime 
\begin{equation}\label{eq:RGG_critical_scaling}
\frac d{n^3}\longrightarrow\lambda\in(0,\infty).
\end{equation}
The above setting is slightly different from \citet{racz2019smooth} in that the points $x_1, \dots, x_n$ are spherical rather than Gaussian and the diagonal of the observed matrix is zero. Nevertheless, these differences do not change the asymptotic normality as shown by the following result. 

\begin{proposition}[Adaptation of \cite{racz2019smooth}]
\label{prop:parent-asymptotics}
Suppose \eqref{eq:RGG_critical_scaling} holds. Then
\begin{equation}\label{eq:parent-log-limits}
\log\frac{\dd P_{n,d}}{\dd Q_n}(Y)
\convd
\cN\!\inparen{\pm\frac{1}{12\lambda},\,\frac{1}{6\lambda}},
\end{equation}
where ``$+$'' is for \(Y\sim P_{n,d}\) and ``$-$'' is for
\(Y\sim Q_n\). Consequently,
\begin{equation}\label{eq:parent-tv}
\TV(P_{n,d},Q_n)
\longrightarrow
2\Phi\!\left(\frac{1}{2\sqrt{6\lambda}}\right)-1.
\end{equation}
\end{proposition}

\noindent
We give a self-contained proof of Proposition~\ref{prop:parent-asymptotics} in Appendix~\ref{sec:continuous}.

\subsection{Random geometric graphs}
\label{sec:rgg}\label{sec:setup}

We now consider detection of geometry in random graphs at critical thresholds. 
Fix an edge density \(p\in(0,1)\). 
Let $A$ be the adjacency matrix of a simple graph we observe. We say that $A$ is a \emph{random geometric graph} and write $A \sim P_{n,d}^{\mathsf{G}}$ (``$\mathsf{G}$'' for graph) if for $1 \le i < j \le n$, 
\begin{equation}
A_{ij}=\1\{Y_{ij} \ge \tau_{d,p}\}, \label{eq:threshold-tau-p}
\end{equation}
where $Y \sim P_{n,d}$ and
\(\tau_{d,p}\) is chosen so that each edge is present with
probability \(p\). Moreover, let \(Q_n^{\mathsf{G}}\) be the distribution of 
the Erd\H{o}s--R\'enyi graph \(G(n,p)\).  Thus $P_{n,d}^{\mathsf{G}}$ and $Q_n^{\mathsf{G}}$ have the same
edge density, while only \(P_{n,d}^{\mathsf{G}}\) has dependence between edges induced by the geometry of the
latent points.

To facilitate stating the results, we define 
\begin{equation*}
z_p=\Phi^{-1}(1-p),\qquad\text{and}\qquad
\rho_p=\frac{\phi(z_p)}{\sqrt{p(1-p)}}.
\end{equation*}
Note that $z_p$ is the $(1-p)$th quantile of $\cN(0,1)$. Therefore, if $A \sim Q_n^{\mathsf{G}}$ and \(Y_{ij}\stackrel{\text{i.i.d.}}{\sim}\cN(0,1)\) for $1 \le i < j \le n$, we can couple them so that 
\begin{equation}
A_{ij} = \1\{Y_{ij} \ge z_p\}. \label{eq:threshold-z-p}
\end{equation}
Moreover, if 
\begin{equation}\label{eq:standardized-edge}
\bar A_{ij} = \frac{A_{ij}-p}{\sqrt{p(1-p)}}
\end{equation}
denotes the standardized edge, then 
\begin{equation}
\rho_p = \E\left[ Y_{ij} \cdot \bar A_{ij} \right] 
= \operatorname{Corr}(Y_{ij}, A_{ij}) \label{eq:rho-p-correlation}
\end{equation}
is the correlation between the latent Gaussian \(Y_{ij}\) and the edge $A_{ij}$.

Therefore, the distributions $P_{n,d}^{\mathsf{G}}$ and $Q_n^{\mathsf{G}}$ are obtained respectively from $P_{n,d}$ and $Q_n$ via thresholding. This operation was used in the original work by \citet{bubeck2016testing} to obtain the indistinguishability of a random geometric graph from an Erd\H{o}s--R\'enyi graph above the threshold $d \asymp n^3$. 
We now show that the thresholding operation not only preserves the order $d \asymp n^3$ but also a stronger statement in the critical regime
\eqref{eq:RGG_critical_scaling}---the asymptotic normality of the log-likelihood
ratio between $P_{n,d}^{\mathsf{G}}$ and $Q_n^{\mathsf{G}}$ analogous to Proposition~\ref{prop:parent-asymptotics}.

\begin{theorem}\label{thm:RGG_ER_LLR_asymptotics}
Suppose \eqref{eq:RGG_critical_scaling} holds and $p$ is a constant.
Then  
\begin{equation}\label{eq:RGG_ER_LLR_asymptotics}
\log\frac{\dd P_{n,d}^{\mathsf{G}}}{\dd Q_n^{\mathsf{G}}}(A)
\convd
\cN\!\inparen{ \pm \frac{\rho_p^6}{12\lambda}, \, \frac{\rho_p^6}{6\lambda}  },
\end{equation}
where ``$+$'' is for $A \sim P_{n,d}^{\mathsf{G}}$ and ``$-$'' is for $A \sim Q_n^{\mathsf{G}}$. Consequently, 
\begin{equation}\label{eq:tv-main}
\TV(P_{n,d}^{\mathsf{G}},Q_n^{\mathsf{G}})
\longrightarrow
2\Phi\!\left(\frac{\rho_p^3}{2\sqrt{6\lambda}}\right)-1.
\end{equation}
\end{theorem}

The proof of Theorem~\ref{thm:RGG_ER_LLR_asymptotics} is deferred to Section~\ref{sec:assembly}. Since $P_{n,d}^{\mathsf{G}}$ and $Q_n^{\mathsf{G}}$ are obtained respectively from $P_{n,d}$ and $Q_n$ via thresholding, by the data processing inequality, it is not surprising that the asymptotic variance of the log-likelihood ratio and the limiting total variation distance are both smaller than their counterparts in Proposition~\ref{prop:parent-asymptotics}. More interestingly, these \emph{global} quantities differ by factors that depend on the \emph{entrywise} correlation $\rho_p \in (0,1)$. This phenomenon is explained by the classical theory of asymptotic normality under information loss \citep{davies1985asymptotic,lecam1988preservation}, which is the backbone of our proof in Section~\ref{sec:proofs-main}. 

Moreover, Theorem~\ref{thm:RGG_ER_LLR_asymptotics} characterizes the optimal asymptotic error for testing between $P_{n,d}^{\mathsf{G}}$ and $Q_n^{\mathsf{G}}$ because, by the Neyman--Pearson lemma,
\begin{equation}
\inf_{\psi} \left( P_{n,d}^{\mathsf{G}}\!\insquare{ \psi(A) = 0  } + Q_n^{\mathsf{G}}\!\insquare{ \psi(A) = 1  } \right) = 1 - \TV(P_{n,d}^{\mathsf{G}},Q_n^{\mathsf{G}}), \label{eq:relation-TV-testing}
\end{equation}
where the infimum is over all tests $\psi$ that returns $1$ (resp.~$0$) if the test result is $A \sim P_{n,d}^{\mathsf{G}}$ (resp.~$A \sim Q_n^{\mathsf{G}}$). 
\cite{bubeck2016testing} proposed the \emph{signed triangle count} as a computationally efficient statistic to achieve consistent testing for $d \ll n^3$. 
We now show that it in fact achieves the optimal asymptotic error throughout the critical regime \eqref{eq:RGG_critical_scaling}. 

Using the standardized edges
\(\bar A_{ij}\) as in \eqref{eq:standardized-edge}, define the (normalized) signed triangle count
\begin{equation}\label{eq:signedtriangle_def}
S_n^{\mathsf{G}}(A) := \frac{1}{\sqrt{\binom n3}} \sum_{1 \le i<j<k \leq n} \bar A_{ij}\bar A_{ik}\bar A_{jk}.
\end{equation}
Define the test $\varphi_n : \inbraces{0,1}^{\binom{[n]}{2}} \to \inbraces{0,1}$ by
\begin{equation}
\varphi_n(A) = \boldsymbol{1}\!\inbraces{ S_n^{\mathsf{G}}(A) > \frac{\rho_p^3}{2\sqrt{6 \lambda}} }.
\label{eq:RGG_ER_signedTriangle_test}
\end{equation}
This threshold test achieves the following asymptotic error, proved in Section~\ref{sec:assembly}.

\begin{theorem}\label{thm:RGG_ER_optimal_test}
Suppose \eqref{eq:RGG_critical_scaling} holds and $p$ is a constant. The test \eqref{eq:RGG_ER_signedTriangle_test} satisfies
\begin{equation*}
P_{n,d}^{\mathsf{G}}\!\insquare{ \varphi_n(A) = 0  } + Q_n^{\mathsf{G}}\!\insquare{ \varphi_n(A) = 1  } \longrightarrow 2 \Phi\!\inparen{  -\frac{\rho_p^3}{2\sqrt{6\lambda} }  }.
\end{equation*}
\end{theorem}
Combining \eqref{eq:tv-main}, \eqref{eq:relation-TV-testing}, 
and Theorem~\ref{thm:RGG_ER_optimal_test} shows that the signed
triangle test $\varphi_n$ is asymptotically optimal. 

\subsection{Rank statistics and kernel matrices}
\label{sec:ranks}\label{subsec:rank-results}

Next, we consider the rank statistics of $Y$ from $P_{n,d}$ or $Q_n$ given at the beginning of Section~\ref{sec:main-results}. Note that there are no ties in $(Y_{ij})_{1\le i<j\le n}$ almost surely under both models. 
Let 
\begin{equation}
\xi_1<\cdots<\xi_{\binom{n}{2}} \label{eq:def-xi-order-statistics}
\end{equation}
be the order statistics of $(Y_{ij})_{1\le i<j\le n}$. 
The rank statistics can be equivalently represented by the bijection 
\begin{equation*}
\Pi:\binom{[n]}{2}\longrightarrow\left[\binom{n}{2}\right] 
\end{equation*}
such that 
$$
Y_{ij}=\xi_{\Pi(i,j)} .
$$
Let \(P_{n,d}^{\mathsf{R}}\) (``$\mathsf{R}$'' for rank) be the distribution of \(\Pi\) for $Y \sim P_{n,d}$. Let \(Q_n^{\mathsf{R}}\) be the distribution of \(\Pi\) for $Y \sim Q_n$, that is, the uniform distribution on bijections from
\(\binom{[n]}{2}\) to \(\left[\binom{n}{2}\right]\).

Our next result is the asymptotic normality of the log-likelihood ratio between
\(P_{n,d}^{\mathsf{R}}\) and \(Q_n^{\mathsf{R}}\) in the critical regime
\eqref{eq:RGG_critical_scaling}.

\begin{theorem}
\label{thm:rank-lan}
Suppose \eqref{eq:RGG_critical_scaling} holds. Then
\begin{equation}\label{eq:rank-log-limits}
\log\frac{\dd P_{n,d}^{\mathsf{R}}}{\dd Q_n^{\mathsf{R}}}(\Pi)
\convd
\cN\!\inparen{\pm\frac{1}{12\lambda},\,\frac{1}{6\lambda}},
\end{equation}
where ``$+$'' is for $\Pi\sim P_{n,d}^{\mathsf{R}}$ and
``$-$'' is for $\Pi\sim Q_n^{\mathsf{R}}$. Consequently,
\begin{equation}\label{eq:rank-tv}
\TV\!\left(P_{n,d}^{\mathsf{R}},Q_n^{\mathsf{R}}\right)
\longrightarrow
2\Phi\!\left(\frac1{2\sqrt{6\lambda}}\right)-1.
\end{equation}
\end{theorem}

Theorem~\ref{thm:rank-lan} is proved in Section~\ref{subsec:rank-assembly}. 
The rank statistics from \(P_{n,d}^{\mathsf{R}}\) or \(Q_n^{\mathsf{R}}\) are \emph{coarsened} ordinal data compared to the cardinal data in the matrix from \(P_{n,d}\) or \(Q_n\). 
However, the limits in Theorem~\ref{thm:rank-lan} are exactly the same as those in Proposition~\ref{prop:parent-asymptotics}, showing that there is no information loss. 
The phenomenon that ranks preserve information is classical in statistics and has been observed in other settings \citep{hallin1988optimal,bickel1997local,hoff2014information}.

Similar to the graph case, we also have a computationally efficient test
that achieves the optimal asymptotic error using rank statistics. For
\(1\le r\le\binom{n}{2}\), define the Gaussian quantiles
\begin{equation}\label{eq:rank-gaussian-quantiles}
g_r=\Phi^{-1}\bigg(\frac r{\binom{n}{2}+1}\bigg).
\end{equation}
Using Gaussian quantiles, we define the normalized cubic
rank statistic 
\begin{equation}\label{eq:cubic-rank-statistic}
S_n^{\mathsf{R}}(\Pi)
:=\frac1{\sqrt{\binom n3}}\sum_{1\le i<j<k\le n}
g_{\Pi(i,j)}g_{\Pi(i,k)}g_{\Pi(j,k)},
\end{equation}
and define the test
\begin{equation}\label{eq:rank-test}
\psi_n(\Pi)=\1\!\left\{
S_n^{\mathsf{R}}(\Pi)>\frac{1}{2\sqrt{6\lambda}}
\right\}.
\end{equation}
The threshold test selects $P_{n,d}^{\mathsf{R}}$ (resp.~$Q_n^{\mathsf{R}}$) when \(\psi_n(\Pi)=1\) (resp.~\(\psi_n(\Pi)=0\)).
It achieves the following asymptotic error, proved in
Section~\ref{subsec:rank-assembly}.

\begin{theorem}\label{thm:rank-optimal-test}
Suppose \eqref{eq:RGG_critical_scaling} holds. Then the test
\eqref{eq:rank-test} satisfies
\begin{equation*}
P_{n,d}^{\mathsf{R}}\!\insquare{\psi_n(\Pi)=0}
+Q_n^{\mathsf{R}}\!\insquare{\psi_n(\Pi)=1}
\longrightarrow
2\Phi\!\left(-\frac1{2\sqrt{6\lambda}}\right).
\end{equation*}
\end{theorem}
Combining \eqref{eq:rank-tv}, \eqref{eq:relation-TV-testing}, and
Theorem~\ref{thm:rank-optimal-test} shows that the test
\(\psi_n\) is asymptotically optimal.

The above results have the following direct consequence in detecting geometry in a random kernel matrix with an
\emph{unknown} strictly increasing kernel \(\kappa:\R\to\R\). (The case that $\kappa$ is strictly decreasing can be treated similarly.) 
Recall the matrix \(Y \sim P_{n,d}\) or \(Y \sim Q_n\) from the
beginning of Section~\ref{sec:main-results}. Suppose we observe
\(\kappa(Y)\), where \(\kappa\) is applied entrywise. 
Since $Y$ and $\kappa(Y)$ have the same rank statistics, the bijection $\Pi$ defined at the beginning of Section~\ref{subsec:rank-results} can be computed from $\kappa(Y)$. We immediately obtain the following.

\begin{corollary}
\label{cor:kernel-testing}
Suppose \eqref{eq:RGG_critical_scaling} holds.
Consider the problem of distinguishing \(Y\sim P_{n,d}\) from \(Y\sim Q_n\) using the observation \(\kappa(Y)\), where \(\kappa:\R\to\R\) is strictly increasing and is allowed to depend on $n$. 
Then the test \(\psi_n(\Pi)\), which does not require the knowledge of $\kappa$, is asymptotically optimal among all
tests based on \(\kappa(Y)\), even those that know \(\kappa\).
\end{corollary}

\begin{proof}
Theorem~\ref{thm:rank-optimal-test} gives the asymptotic error for the test \(\psi_n(\Pi)\). Even if $\kappa$ were known, observing
\(\kappa(Y)\) would be equivalent to observing \(Y\). Therefore, Proposition~\ref{prop:parent-asymptotics} together with \eqref{eq:relation-TV-testing} shows the claimed optimality.
\end{proof}

\section{Proofs of main results}\label{sec:proofs-main}

We first state in Section~\ref{sec:coarsening} the general result of \cite{davies1985asymptotic} for asymptotic normality under information loss. We then study random graphs and rank statistics in Sections~\ref{sec:assembly} and~\ref{subsec:rank-assembly} respectively.

\subsection{Likelihood ratio for coarsened observations}\label{sec:coarsening}

For probability distributions \(\mu,\nu\) on \(\R\), we use the \emph{bounded Lipschitz distance} (see, e.g., \citet{dudley2002real})
\[
d_{\BL}(\mu,\nu)=
\sup_{\substack{\|f\|_\infty\le1\\\operatorname{Lip}(f)\le1}}
\left|\int f\dd\mu-\int f\dd\nu\right|.
\]
It is classical that convergence in $d_{\BL}$ is equivalent to convergence in distribution. 
We reformulate the result of \citet{davies1985asymptotic} below and provide a self-contained proof in Appendix~\ref{app:coarsening-proof}.

\begin{theorem}[Adaptation of \citet{davies1985asymptotic}, Theorem 4.1]\label{thm:coarsening}
Let \(P_n\ll Q_n\) on a probability space \((\Omega_n,\cF_n)\) with likelihood ratio
\(L_n=\dd P_n/\dd Q_n\). Consider $\sigma$-algebras \(\cG_n\subseteq\cF_n\).  Suppose there are
constants \(0\le J\le I\), a \(\cG_n\)-measurable random variable \(M_n\), and a
random variable \(R_n\) such
that the following hold:
\begin{enumerate}
\item 
As $n \to \infty$, we have 
\begin{equation}\label{eq:coarsening-expansion}
\log L_n=M_n+R_n-\frac{I}{2}+o_{Q_n}(1),
\end{equation}
where $o_{Q_n}(1)$ denotes a random variable that converges to $0$ in $Q_n$-probability. 
\item The random variables $M_n$ converge as 
\begin{equation}\label{eq:coarsening-M}
M_n\convd\cN(0,J).
\end{equation}
\item The conditional distribution of $R_n$ given $\cG_n$, denoted by $\cL_{Q_n}(R_n\mid\cG_n)$, converges as
\begin{equation}\label{eq:coarsening-conditional}
d_{\BL}\!\left(\cL_{Q_n}(R_n\mid\cG_n),
\cN(0,I-J)\right) = o_{Q_n}(1).
\end{equation}
\end{enumerate}
Then, for \(L_n^{\cG}:=\E_{Q_n}[L_n\mid\cG_n]\),
\begin{equation}\label{eq:coarsening-conclusion}
\log L_n^{\cG}=M_n-\frac{J}{2}+o_{Q_n}(1).
\end{equation}
\end{theorem}

The random variable \(M_n\) represents a dominating statistic visible from
\(\cG_n\).  In the Gaussian limit, integrating out the residual $\exp(R_n)$ in the likelihood ratio $L_n$ contributes
a factor \(\exp((I-J)/2)\), changing the constant term in the log-likelihood ratio from
\(-I/2\) in \eqref{eq:coarsening-expansion} to \(-J/2\) in \eqref{eq:coarsening-conclusion}. 

\subsection{Proofs for random geometric graphs}
\label{sec:assembly}

We apply Theorem~\ref{thm:coarsening} to prove
Theorem~\ref{thm:RGG_ER_LLR_asymptotics}. 
Let us first prove a strong result---an expansion of the log-likelihood ratio. 

\begin{theorem}\label{thm:RGG_ER_LLR_strongerSignedTriangle_asymptotics}
In the setting of Theorem~\ref{thm:RGG_ER_LLR_asymptotics}, let $S_n^{\mathsf{G}}(A)$ be given by \eqref{eq:signedtriangle_def}. Then for $A \sim Q_n^{\mathsf{G}}$,
\begin{equation*}
\log\frac{\dd P_{n,d}^{\mathsf{G}}}{\dd Q_n^{\mathsf{G}}}(A)
=\frac{\rho_p^3}{\sqrt{6\lambda}}S_n^{\mathsf{G}}(A)
-\frac{\rho_p^6}{12\lambda}+o_{Q_n^{\mathsf{G}}}(1).
\end{equation*}
\end{theorem}
This result identifies the signed triangle count as an asymptotically optimal statistic. The asymptotic normality of $S_n^{\mathsf{G}}(A)$ itself (c.f.~Proposition \ref{prop:signed-triangle-clt}) then leads to Theorem \ref{thm:RGG_ER_LLR_asymptotics}.

Recall that $P_{n,d}^{\mathsf{G}}$ and $Q_n^{\mathsf{G}}$ are obtained from \(P_{n,d}\) and
\(Q_n\) by thresholding the matrix entries at $\tau_{d,p}$ and $z_p$ respectively, as in \eqref{eq:threshold-tau-p} and \eqref{eq:threshold-z-p}. 
To address the small discrepancy between the thresholds $\tau_{d,p}$ and $z_p$, we let $\widetilde P_{n,d}^{\mathsf{G}}$ denote the distribution of $A$ with $A_{ij}$ given by \eqref{eq:threshold-z-p} where $Y \sim P_{n,d}$. 
In other words, $\widetilde P_{n,d}^{\mathsf{G}}$ is the distribution of the random geometric graph obtained from thresholding $Y$ at the Gaussian quantile $z_p$. 
Lemma~\ref{lem:rgg-threshold-correction} shows that $\TV(\widetilde P_{n,d}^{\mathsf{G}},P_{n,d}^{\mathsf{G}})=O(n^2/d)=o(1)$ in the regime of interest \eqref{eq:RGG_critical_scaling}. 
Therefore, we can work with the distribution $\widetilde P_{n,d}^{\mathsf{G}}$ below.

The application of Theorem~\ref{thm:coarsening} starts with the following simple result. 

\begin{lemma}\label{lem:pushforward}
For \(d\ge n+1\), almost surely under \(Q_n^{\mathsf{G}}\),
\begin{equation}\label{eq:pushforward}
\frac{\dd \widetilde P_{n,d}^{\mathsf{G}}}{\dd Q_n^{\mathsf{G}}}(A)
=\E_{Q_n}\left[
\frac{\dd P_{n,d}}{\dd Q_n}(Y)
\ \middle|\ A\right].
\end{equation}
\end{lemma}

\begin{proof}
For \(d\ge n+1\), Lemma~\ref{lem:exact-cont-likelihood} gives
\(P_{n,d}\ll Q_n\).  If \(f\) is any bounded
measurable function of \(A\), then the tower property and a change of measure give
\[
\E_{Q_n^{\mathsf{G}}}\!\left[f(A)
\E_{Q_n}\!\left[
\frac{\dd P_{n,d}}{\dd Q_n}(Y)
\middle| A\right]\right]
=\E_{Q_n}\!\left[
f(A)\frac{\dd P_{n,d}}{\dd Q_n}(Y)\right]
=\E_{P_{n,d}}[f(A)]
=\E_{\widetilde P_{n,d}^{\mathsf{G}}}[f(A)].
\]
Thus the conditional expectation is the Radon--Nikodym derivative, which proves \eqref{eq:pushforward}.
\end{proof}

We now verify the hypotheses of Theorem~\ref{thm:coarsening}. 
The first ingredient is a large-$n$ approximation of the log-likelihood ratio between the parent distributions $P_{n,d}$ and $Q_n$. Define the
cubic statistic
\begin{equation*}
S_n(Y):=\frac{1}{\sqrt{\binom n3}}\sum_{1\le i<j<k\le n}Y_{ij}Y_{ik}Y_{jk}.
\end{equation*}

\begin{theorem}\label{thm:continuous-lan}
Suppose \eqref{eq:RGG_critical_scaling} holds. Then, for $Y\sim Q_n$,
\begin{equation}\label{eq:continuous-lan}
\log\frac{\dd P_{n,d}}{\dd Q_n}(Y)
=\frac{S_n(Y)}{\sqrt{6\lambda}}-\frac{1}{12\lambda}
+o_{Q_n}(1).
\end{equation}
\end{theorem}

\noindent
We prove Theorem~\ref{thm:continuous-lan} and show that Proposition~\ref{prop:parent-asymptotics} follows from it in Appendix~\ref{sec:continuous}.

The next step is to split $S_n(Y)/\sqrt{6\lambda}$ into a part
$M_n$ that is measurable with respect to $A$ and a residual $R_n$ as in Theorem~\ref{thm:coarsening}. The split we use is
\begin{equation}
\frac{S_n(Y)}{\sqrt{6\lambda}}
=\underbrace{\frac{\rho_p^3 S_n^{\mathsf{G}}(A)}{\sqrt{6\lambda}}}_{=:M_n}
+\underbrace{\frac{S_n(Y)-\rho_p^3 S_n^{\mathsf{G}}(A)}{\sqrt{6\lambda}}}_{=:R_n}. \label{eq:split-M_n-R_n}
\end{equation}
Then the second hypothesis of Theorem~\ref{thm:coarsening} is precisely the asymptotic normality of the signed triangle count, which is a special case of Theorem~1 of \citet{janson1994orthogonal}.

\begin{proposition}[\citet{janson1994orthogonal}]
\label{prop:signed-triangle-clt}
Under $A \sim Q_n^{\mathsf{G}}$,
\begin{equation*}
S_n^{\mathsf{G}}(A)
\convd
\cN(0,1).
\end{equation*}
\end{proposition}

Finally, the last hypothesis of Theorem~\ref{thm:coarsening} is verified using the following proposition, proved in Appendix~\ref{sec:residual}.

\begin{proposition}
\label{prop:quenched-residual}
In $Q_n^{\mathsf{G}}$-probability,
\begin{equation*}
d_{\BL}\inparen{
\cL_{Q_n}\inparen{S_n(Y)-\rho_p^3 S_n^{\mathsf{G}}(A)\,\middle|\, A}, \,
\cN\!\inparen{0,1-\rho_p^6}}
\longrightarrow0.
\end{equation*}
\end{proposition}

\noindent
With these auxiliary results in position, the main theorems follow by application of Theorem~\ref{thm:coarsening}.

\medskip

\begin{proof}[Proof of Theorem \ref{thm:RGG_ER_LLR_strongerSignedTriangle_asymptotics}]
Apply Theorem~\(\ref{thm:coarsening}\) by taking
\(P_{n,d}\) and \(Q_n\) as the parent distributions, with
\(\mathcal G_n=\sigma(A)\), $M_n$ and $R_n$ defined in \eqref{eq:split-M_n-R_n}, 
\(I=\frac{1}{6\lambda}\), and
\(J=\frac{\rho_p^6}{6\lambda}\). 
Lemma~\(\ref{lem:pushforward}\) identifies the conditional likelihood ratio with that between the random graph models $\widetilde P_{n,d}^{\mathsf{G}}$ and $Q_n^{\mathsf{G}}$. 
Theorem~\(\ref{thm:continuous-lan}\),
Proposition~\(\ref{prop:signed-triangle-clt}\), and
Proposition~\(\ref{prop:quenched-residual}\) verify the three required
hypotheses, in that order. 
This yields
\[
\log \widetilde L_n(A)=\frac{\rho_p^3 S_n^{\mathsf{G}}(A)}{\sqrt{6\lambda}}-\frac J2+o_{Q_n^{\mathsf{G}}}(1),
\]
where $\widetilde L_n := \dd \widetilde P_{n,d}^{\mathsf{G}} / \dd Q_n^{\mathsf{G}}$. 
By Proposition~\ref{prop:signed-triangle-clt}, $\widetilde L_n$ converges in distribution to a log-normal random variable, so $1/\widetilde L_n=O_{Q_n^{\mathsf{G}}}(1)$. 
On the other hand, Lemma~\ref{lem:rgg-threshold-correction} gives, for $L_n := \dd P_{n,d}^{\mathsf{G}}/\dd Q_n^{\mathsf{G}}$,
\[
\E_{Q_n^{\mathsf{G}}}|L_n-\widetilde L_n|=2\TV(P_{n,d}^{\mathsf{G}},\widetilde P_{n,d}^{\mathsf{G}})=o(1), 
\]
so by Markov's inequality, $L_n-\widetilde L_n=o_{Q_n^{\mathsf{G}}}(1)$. Consequently,
\[
\frac{L_n}{\widetilde L_n}-1
=\frac{L_n-\widetilde L_n}{\widetilde L_n}
=o_{Q_n^{\mathsf{G}}}(1).
\]
Continuity of the logarithm at $1$ then gives
\[
\log L_n-\log\widetilde L_n
=\log\frac{L_n}{\widetilde L_n}
=o_{Q_n^{\mathsf{G}}}(1),
\]
proving Theorem \ref{thm:RGG_ER_LLR_strongerSignedTriangle_asymptotics}.
\end{proof}

We are now ready to prove the main results for random graphs using Theorem \ref{thm:RGG_ER_LLR_strongerSignedTriangle_asymptotics}.

\medskip

\begin{proof}[Proof of Theorem \ref{thm:RGG_ER_LLR_asymptotics}]
By Proposition~\ref{prop:signed-triangle-clt} and Theorem \ref{thm:RGG_ER_LLR_strongerSignedTriangle_asymptotics}, we immediately obtain \eqref{eq:RGG_ER_LLR_asymptotics} for $A \sim Q_n^{\mathsf{G}}$.
Le Cam's third lemma (Lemma~\ref{lem:le-cam-third}) yields the corresponding statement in \eqref{eq:RGG_ER_LLR_asymptotics} for $A \sim P_{n,d}^{\mathsf{G}}$ by considering the limiting joint distribution of $\inparen{\log \frac{\ud P_{n,d}^{\mathsf{G}}}{\ud Q_n^{\mathsf{G}}}, \log \frac{\ud P_{n,d}^{\mathsf{G}}}{\ud Q_n^{\mathsf{G}}}}$ under $Q_n^{\mathsf{G}}$. 

The TV limit \eqref{eq:tv-main} follows from the Neyman--Pearson lemma which asserts that the optimal test between two simple hypotheses (i.e., the minimizer in \eqref{eq:relation-TV-testing}) is achieved by thresholding the log-likelihood ratio at zero. 
\end{proof}

\begin{proof}[Proof of Theorem \ref{thm:RGG_ER_optimal_test}]
Proposition~\ref{prop:signed-triangle-clt} gives
$
S_n^{\mathsf{G}}(A)
\convd\cN(0,1)
$
under 
$Q_n^{\mathsf{G}}$.
Hence
\[
Q_n^{\mathsf{G}}\insquare{\varphi_n=1}
\longrightarrow
\Phi\inparen{-\frac{\rho_p^3}{2\sqrt{6\lambda}}}.
\]
In addition, Theorem~\ref{thm:RGG_ER_LLR_strongerSignedTriangle_asymptotics}
gives, under \(Q_n^{\mathsf{G}}\),
\[
\inparen{
S_n^{\mathsf{G}}(A),
\log\frac{\dd P_{n,d}^{\mathsf{G}}}{\dd Q_n^{\mathsf{G}}}}
\convd
\inparen{
Z,
\frac{\rho_p^3}{\sqrt{6\lambda}}Z
-\frac{\rho_p^6}{12\lambda}},
\]
where $Z\sim\cN(0,1)$. 
It then follows by an application of Le Cam's third lemma
(Lemma~\ref{lem:le-cam-third}) that under \(P_{n,d}^{\mathsf{G}}\),
\begin{equation*}
S_n^{\mathsf{G}}(A)
\convd
\cN\inparen{\frac{\rho_p^3}{\sqrt{6\lambda}},1}.
\end{equation*}
As a result, 
\[
P_{n,d}^{\mathsf{G}}\insquare{\varphi_n=0}
\longrightarrow
\Phi\inparen{-\frac{\rho_p^3}{2\sqrt{6\lambda}}}.
\]
Adding the two errors proves the theorem.
\end{proof}

\subsection{Proofs for rank statistics}
\label{subsec:rank-assembly}

We apply Theorem~\ref{thm:coarsening} to prove
Theorems~\ref{thm:rank-lan} and~\ref{thm:rank-optimal-test}.
Recall the cubic rank statistic \(S_n^{\mathsf{R}}(\Pi)\) defined in
\eqref{eq:cubic-rank-statistic}. As in the graph case, our next result gives a stronger
log-likelihood ratio expansion in the critical regime
\eqref{eq:RGG_critical_scaling}, identifying \(S_n^{\mathsf{R}}(\Pi)\)
as an asymptotically optimal statistic.

\begin{theorem}
\label{thm:rank-likelihood-expansion}
Suppose \eqref{eq:RGG_critical_scaling} holds. Under
\(Q_n^{\mathsf{R}}\),
\begin{equation*}
\log
\frac{\dd P_{n,d}^{\mathsf{R}}}
{\dd Q_n^{\mathsf{R}}}(\Pi)
=
\frac{S_n^{\mathsf{R}}(\Pi)}{\sqrt{6\lambda}}
-\frac{1}{12\lambda}
+o_{Q_n^{\mathsf{R}}}(1).
\end{equation*}
\end{theorem}

A critical ingredient in the proof is the following lemma, which shows that the cubic rank statistic $S_n^{\mathsf{R}}(\Pi)$ can be used to reconstruct the cubic polynomial $S_n(Y)$ in $L^2$. It is relatively straightfoward to show that any individual $Y_{ij}$ can be reconstructed from its rank $\Pi(i,j)$ accurately by approximation by \(g_{\Pi(i,j)}\), the Gaussian quantile
defined in \eqref{eq:rank-gaussian-quantiles}. The main technical difficulty is to  extend this approximation to \(S_n(Y)\),
which sums the products \(Y_{ij}Y_{ik}Y_{jk}\) over all such triangles, and show that the errors in the triangle products do not accumulate across the sum.

\begin{lemma}
\label{lem:rank-statistic-reconstruction}
Under $Q_n$,
\begin{equation}\label{eq:rank-statistic-reconstruction}
\E_{Q_n}\left[
\inparen{S_n(Y)-S_n^{\mathsf{R}}(\Pi)}^2
\right]
\longrightarrow0.
\end{equation}
\end{lemma}

\noindent 
The proof of this lemma is deferred to
Appendix~\ref{subsec:rank-reconstruction}. The next result provides the asymptotic normality of the cubic rank statistic.
\begin{proposition}
\label{prop:cubic-rank-statistic-clt}
Under
\(Q_n^{\mathsf{R}}\),
\begin{equation*}
S_n^{\mathsf{R}}(\Pi)
\convd\cN(0,1).
\end{equation*}
\end{proposition}

\begin{proof}
By Lemma~\ref{lem:cubic-statistic-clt}, \(S_n(Y)\convd\cN(0,1)\) under \(Q_n\), so the result follows by Lemma~\ref{lem:rank-statistic-reconstruction}.
\end{proof}

We now give the formal link between the parent distributions $(P_{n,d}, Q_n)$ and the distributions of rank statistics $(P_{n,d}^{\mathsf{R}}, Q_n^{\mathsf{R}})$. By the same change-of-measure argument as in Lemma~\ref{lem:pushforward},
with \(A\) replaced by \(\Pi\), we obtain that for \(d\ge n+1\), 
\begin{equation}\label{eq:rank-pushforward-likelihood}
\frac{\dd P_{n,d}^{\mathsf{R}}}
{\dd Q_n^{\mathsf{R}}}(\Pi)
=
\E_{Q_n}\left[
\frac{\dd P_{n,d}}{\dd Q_n}(Y)
\,\middle|\,\Pi
\right].
\end{equation}

\medskip

\begin{proof}[Proof of Theorem~\ref{thm:rank-likelihood-expansion}]
Apply Theorem~\ref{thm:coarsening} by taking \(P_{n,d}\) and \(Q_n\)
as the parent distributions, with \(\mathcal G_n=\sigma(\Pi)\),
\(I=J=1/(6\lambda)\), and
\[
M_n=\frac{S_n^{\mathsf{R}}(\Pi)}{\sqrt{6\lambda}},\qquad\text{and}\qquad
R_n=\frac{S_n(Y)-S_n^{\mathsf{R}}(\Pi)}{\sqrt{6\lambda}}.
\]
Theorem~\ref{thm:continuous-lan} gives the required expansion of the parent log-likelihood.
The variable $M_n$ is $\mathcal G_n$-measurable, and
Proposition~\ref{prop:cubic-rank-statistic-clt} gives $M_n\convd\cN(0,J)$.
Lemma~\ref{lem:rank-statistic-reconstruction} gives \(\|R_n\|_2\to0\).
In particular,
\[
\begin{aligned}
d_{\BL}\!\left(\cL_{Q_n}(R_n\mid\Pi),\delta_0\right)
&=\sup_{\|f\|_\infty\le1,\,\operatorname{Lip}(f)\le1}
\left|\E_{Q_n}[f(R_n)\mid\Pi]-f(0)\right|\\
&\le\sup_{\|f\|_\infty\le1,\,\operatorname{Lip}(f)\le1}
\E_{Q_n}\!\left[|f(R_n)-f(0)|\mid\Pi\right]\\
&\le\E_{Q_n}\!\left[|R_n|\mid\Pi\right]
\le\left(\E_{Q_n}\!\left[R_n^2\mid\Pi\right]\right)^{1/2}.
\end{aligned}
\]
To see the bound converges to zero in \(Q_n\)-probability, note that for every \(\varepsilon>0\), by Markov's inequality and the tower property, \(Q_n(\E_{Q_n}[R_n^2\mid\Pi]>\varepsilon^2)\) is at most \(\varepsilon^{-2}\E_{Q_n}R_n^2
=\varepsilon^{-2}\|R_n\|_2^2\to 0\).
Combining the above results with \eqref{eq:rank-pushforward-likelihood}, we then apply Theorem~\ref{thm:coarsening} to conclude the proof.
\end{proof}

We are now ready to prove the main results for this model. 

\medskip

\begin{proof}[Proof of Theorem~\ref{thm:rank-lan}]
Theorem~\ref{thm:rank-likelihood-expansion} and Proposition~\ref{prop:cubic-rank-statistic-clt}
give \eqref{eq:rank-log-limits} under \(Q_n^{\mathsf{R}}\).
The same argument as in the proof of
Theorem~\ref{thm:RGG_ER_LLR_asymptotics}, with \(J=1/(6\lambda)\),
gives the limit under \(P_{n,d}^{\mathsf{R}}\) and the TV limit
\eqref{eq:rank-tv}.
\end{proof}

\begin{proof}[Proof of Theorem~\ref{thm:rank-optimal-test}]
Using Theorem~\ref{thm:rank-likelihood-expansion} and Proposition~\ref{prop:cubic-rank-statistic-clt},
the same argument as in the proof of Theorem~\ref{thm:RGG_ER_optimal_test}
gives limiting error \(\Phi(-1/(2\sqrt{6\lambda}))\) under each distribution.
Adding the two errors proves the theorem.
\end{proof}

\appendix

\section{Classical theory}\label{app:coarsening}

\subsection{Asymptotic normality under information loss}\label{app:coarsening-proof}

\begin{proof}[Proof of Theorem~\ref{thm:coarsening}]
We provide a self-contained proof following \citet[Theorem~4.1]{davies1985asymptotic}.

Let $\cE_n=\{L_n>0\}$, define $r_n=\log L_n-M_n-R_n+I/2$ on $\cE_n$, and set $r_n=0$ on $\cE_n^c$. By \eqref{eq:coarsening-expansion}, we have $Q_n(\cE_n)\to1$ and $r_n=o_{Q_n}(1)$.

Using the identity $L_n = \inparen{L_n \wedge K} + (L_n - K)_+$ for any fixed constant $K > 0$, we have
\begin{equation}\label{eq:coarsening_conclusionSplitIntoThreeTerms}
\abs{L_n^{\cG} - \exp\inparen{M_n - \frac{J}{2}}} \leq (\text{I}) + (\text{II}) + (\text{III}),
\end{equation}
where 
\begin{align*}
(\text{I}) &= \abs{ \EE_{Q_n}\insquare{L_n \wedge K \mid \cG_n} - g_K(M_n) } \\
(\text{II}) &= \EE_{Q_n}\insquare{(L_n - K)_+ \mid \cG_n} \\
(\text{III}) &= \abs{g_K(M_n) - \exp\inparen{M_n - \frac{J}{2}}},
\end{align*}
and where we define $g_K(x) = \EE_Z\insquare{ \min\inbraces{\exp\inparen{ x + \sqrt{I-J}Z - \frac{I}{2}  }, K}  }$ for $Z \sim \cN(0,1)$. The quantity $g_K(M_n)$ is a truncated approximation to $\exp(M_n-J/2)$. In other words, term (I) describes the convergence of the truncated likelihood ratio $L_n$, while terms (II) and (III) describe the truncation error in $L_n$ and its approximation respectively.

We consider term (I) first.
Define the function $h_K(u) = e^u \wedge K$. We have
\begin{equation*}
\text{(I)} \leq \text{(I--a)} + \text{(I--b)},
\end{equation*}
where 
\begin{align*}
\text{(I--a)} &= \abs{ \EE_{Q_n} \insquare{ L_n \wedge K  \mid \cG_n } - \EE_{Q_n} \insquare{ h_K\inparen{M_n + R_n - \frac{I}{2}} \mid \cG_n  } } \\
\text{(I--b)} &= \abs{ \EE_{Q_n} \insquare{ h_K\inparen{M_n + R_n - \frac{I}{2}} \,\bigg|\, \cG_n  } - g_K(M_n) }.
\end{align*}
For term $\text{(I--a)}$, notice that on $\cE_n$, we have $L_n \wedge K = h_K(M_n + R_n -\frac{I}{2} + r_n)$. Using that $h_K$ is $K$-Lipschitz and $h_K \in [0,K]$, we have
\begin{align}\label{eq:coarsening_I--a}
\text{(I--a)} \leq K Q_n(\cE_n^c \mid \cG_n) + K \EE_{Q_n}\insquare{ \min\inbraces{\abs{r_n}, 1}  \mid \cG_n }.
\end{align}

We record an elementary observation that will be used repeatedly. For a random variable $X_n \geq 0$, we have the implication
\begin{equation}\label{eq:conditionalMarkovsInequality_generalPrinciple}
\EE_{Q_n} X_n \rightarrow 0 \quad \Longrightarrow \quad \EE_{Q_n}\insquare{X_n \mid \cG_n} = o_{Q_n}(1).  
\end{equation}
Indeed by Markov's inequality, for any $\delta > 0$, $Q_n\insquare{ \EE_{Q_n}\insquare{X_n \mid \cG_n } > \delta  } \leq \frac{\EE_{Q_n}X_n}{\delta} \to 0$.

Set $X_n \equiv \boldsymbol{1}\inbraces{\cE_n^c}$ and since $\EE_{Q_n}X_n = Q_n(\cE_n^c) \to 0$, we have from \eqref{eq:conditionalMarkovsInequality_generalPrinciple} that the first term on the right of \eqref{eq:coarsening_I--a} is $o_{Q_n}(1)$. Next, set $X_n \equiv \min\inbraces{\abs{r_n}, 1}$. Then since $r_n = o_{Q_n}(1)$, $X_n \to 0$ in $Q_n$-probability. But since $X_n$ is also bounded by 1, so it also converges to $0$ in $L^1$. We have by \eqref{eq:conditionalMarkovsInequality_generalPrinciple} again that the second term on the right of \eqref{eq:coarsening_I--a} is also $o_{Q_n}(1)$. Altogether $\text{(I--a)} = o_{Q_n}(1)$, with constants depending on $K$.

We claim that term (I--b) is also $o_{Q_n}(1)$. This is immediate from the hypothesis \eqref{eq:coarsening-conditional} applied to the function $u \mapsto \frac{1}{K}h_K\inparen{M_n + u - \frac{I}{2}}$ which is bounded in sup-norm by 1 and is $1$-Lipschitz and hence a $d_{\text{BL}}$ test function. 
We therefore deduce that $\text{(I)} = o_{Q_n}(1)$.

Controlling term (II) in \eqref{eq:coarsening_conclusionSplitIntoThreeTerms} involves showing that truncation has a negligible effect. More precisely we will show that for any $\epsilon > 0$
\begin{equation}\label{eq:coarsening_termII_conclusion}
\lim_{K\to\infty}\limsup_{n\to\infty} Q_n\insquare{ \EE_{Q_n}\insquare{(L_n - K)_+\mid \cG_n} > \epsilon  } = 0.
\end{equation}  
This will follow by showing uniform integrability of $L_n$. In particular, we will show that
\begin{equation}\label{eq:coarsening_Ln_UI}
\lim_{K\to\infty}\limsup_{n\to\infty}
\E_{Q_n}[(L_n-K)_+]=0.
\end{equation}
Since $L_n \boldsymbol{1}\inbraces{L_n > 2K} \leq 2 (L_n - K)_+$, this proves uniform integrability. With \eqref{eq:coarsening_Ln_UI} in hand, \eqref{eq:coarsening_termII_conclusion} follows by the same principle as in \eqref{eq:conditionalMarkovsInequality_generalPrinciple}.

To show \eqref{eq:coarsening_Ln_UI}, we will elucidate the weak limit $L$ of $L_n$ and show that it also has mean one. We first find the joint limit of $(M_n, R_n)$. Because \(M_n\) is \(\cG_n\)-measurable, conditioning gives, for fixed
\(s,t\in\R\),
\[
\E_{Q_n}\e^{isM_n}
\left(
\E_{Q_n}[\e^{itR_n}\mid\cG_n]
-\e^{-(I-J)t^2/2}
\right)\longrightarrow0.
\]
Indeed, the expression in parentheses converges to zero in probability by
\eqref{eq:coarsening-conditional}, and it is also uniformly bounded, so the
convergence also holds in \(L^1\).  Together with
\eqref{eq:coarsening-M}, this identifies the joint characteristic function
and yields
\[
(M_n,R_n)\convd(G_1,G_2),
\quad\text{where}\quad
G_1\sim\cN(0,J),\quad G_2\sim\cN(0,I-J),\quad\text{and}\quad G_1\perp G_2.
\]
It follows from \eqref{eq:coarsening-expansion} that
\(L_n\convd L \equiv \exp(G-I/2)\), where \(G\sim\cN(0,I)\).  Thus both $L_n$ and $L$ have mean one. Then as $n \to \infty$,
\begin{align*}
\EE_{Q_n}(L_n - K)_+ &= \EE_{Q_n} L_n - \EE_{Q_n} \insquare{L_n \wedge K} \\
&= 1 - \EE_{Q_n} \insquare{L_n \wedge K} \\
&\to 1 - \EE \insquare{L \wedge K} = \EE (L-K)_+.
\end{align*} 
Because $L$ is integrable, $(L-K)_+$ decreases to zero as $K \to \infty$, which proves \eqref{eq:coarsening_Ln_UI}.

We turn to term (III) in \eqref{eq:coarsening_conclusionSplitIntoThreeTerms}. Here we will show for any $\epsilon > 0$
\begin{equation}\label{eq:coarsening_termIII_conclusion}
\lim_{K\to\infty}\limsup_{n\to\infty} Q_n\insquare{  \abs{g_K(M_n) - \exp\inparen{M_n - \frac{J}{2}}} > \epsilon  } = 0.
\end{equation}
To show this, we decompose, for any fixed $A > 0$, 
\begin{equation*}
\begin{aligned}
&Q_n\insquare{\abs{g_K(M_n)-\exp\inparen{M_n-J/2}}>\epsilon}\\
&\quad\le Q_n\insquare{\abs{M_n}>A}
+\boldsymbol{1}\!\inbraces{\sup_{\abs{x}\le A}
\abs{g_K(x)-\exp\inparen{x-J/2}}>\epsilon}.
\end{aligned}
\end{equation*}
Since $M_n\convd\cN(0,J)$, the sequence $M_n$ is tight. We may choose $A > 0$ large enough so that the $\limsup_n Q_n\insquare{\abs{M_n} > A}$ is arbitrarily small. 

Fix such an $A$. By monotone convergence, $g_K(x) \uparrow
\exp\inparen{x-\frac{J}{2}}$. Furthermore, since $g_K$'s and their limit are continuous, Dini's theorem upgrades this to $\sup_{\abs{x} \leq A} \abs{g_K(x) - \exp\inparen{x - \frac{J}{2}}} \to 0$ as $K \to \infty$. Taking first $\limsup_n$ followed by $K \to \infty$ yields \eqref{eq:coarsening_termIII_conclusion}.

Combining the fixed \(K\) estimate \((\mathrm I)=o_{Q_n}(1)\) with the two limiting estimates for (II) and (III), we proceed by first choosing \(K\) sufficiently large such that the latter two errors are small in \(\limsup_{n\to\infty}\). Then, with this \(K\) fixed, letting \(n\to\infty\) to eliminate (I). The decomposition \eqref{eq:coarsening_conclusionSplitIntoThreeTerms} therefore gives
\[
L_n^{\cG}=\exp\inparen{M_n-\frac{J}{2}}+o_{Q_n}(1).
\]Finally, \(M_n\) is tight, so \(\exp\inparen{M_n-J/2}\) is bounded away from zero in \(Q_n\)-probability. Dividing the previous display by this leading term gives the relative approximation \(L_n^{\cG}/\exp\inparen{M_n-J/2}=1+o_{Q_n}(1)\). Taking logarithms and using \(\log(1+o_{Q_n}(1))=o_{Q_n}(1)\) proves \eqref{eq:coarsening-conclusion}.
\end{proof}

\subsection{Hoeffding's criterion for conditional convergence}\label{app:two-replica-proof}

We use the following version of Hoeffding's trick \citep{hoeffding1952permutations}; see
\citet[Lemma~4.1]{dumbgen2013projections}
and \citet[Lemma~2.3]{bucher2019conditional} for this formulation.

\begin{lemma}[\citet{hoeffding1952permutations}]
\label{lem:two-replica-criterion}
For each \(n\), let \(\cG_n\) be a $\sigma$-algebra and let
\(R_n^{(1)},R_n^{(2)}\) be conditionally independent with the same
conditional distribution given \(\cG_n\).  Let \(\nu\) be a fixed probability distribution on
\(\R\).  If
\begin{equation}\label{eq:two-replica-joint-limit}
(R_n^{(1)},R_n^{(2)})
\convd\nu\otimes\nu,
\end{equation}
then
\begin{equation}\label{eq:two-replica-conditional-limit}
d_{\BL}\!\left(
\cL(R_n^{(1)}\mid\cG_n),\nu
\right)\longrightarrow0
\quad\text{in probability}.
\end{equation}
\end{lemma}

\subsection{Le Cam's third lemma}\label{app:le-cam}

Let \(P_n\ll Q_n\) be probability distributions and write \(L_n=\dd P_n/\dd Q_n\). We state the conclusion of Example~6.7 from \citet{vanderVaart1998asymptotic} below.

\begin{lemma}[Le Cam's third lemma]
\label{lem:le-cam-third}
If, under \(Q_n\),
\[
(T_n,\log L_n)\convd
\cN\!\left(
\begin{pmatrix}\mu\\-J/2\end{pmatrix},
\begin{pmatrix}v&c\\c&J\end{pmatrix}
\right),
\]
then \(T_n\convd\cN(\mu+c,v)\) under \(P_n\).
(The limiting Gaussian distribution may be degenerate.)
\end{lemma}

\section{Proofs for the parent models}\label{sec:continuous}

The purpose of this section is to establish Theorem~\ref{thm:continuous-lan}
and Proposition~\ref{prop:parent-asymptotics}.

\subsection{Expanding the parent log-likelihood ratio}

Define the Gram matrix \(R=(\langle x_i,x_j\rangle)_{i,j\le n}\), where $x_i \overset{\text{iid}}{\sim} \Unif(\SSS^{d-1})$. Equivalently, by setting $Y_{ii} \equiv 0$, we have the relation
\begin{equation*}
R=I_n+d^{-1/2}Y.
\end{equation*}
For \(a>(n-1)/2\), define
\[
\Gamma_n(a)=\pi^{n(n-1)/4}
\prod_{j=1}^n\Gamma\!\left(a-\frac{j-1}{2}\right).
\]

We first record an exact expression of the parent log-likelihood ratio.
\begin{lemma}\label{lem:exact-cont-likelihood}
Let $d\ge n+1$. Then, $Q_n$-almost surely,
\begin{align}
\log\frac{\dd P_{n,d}}{\dd Q_n}(Y)
={}&K_{n,d}
+\frac{d-n-1}{2}\log\det(I_n+d^{-1/2}Y)\notag\\
&+\frac14\tr(Y^2)+\log\1\{I_n+d^{-1/2}Y\succ0\},
\label{eq:exact-cont-likelihood-new}
\end{align}
where the right-hand side is understood as $-\infty$ whenever $I_n+d^{-1/2}Y$ is not positive definite, and where
\[
K_{n,d}=\log\frac{\Gamma(d/2)^n}{\Gamma_n(d/2)}
+\frac12\binom n2\log\frac{2\pi}{d}.
\]
\end{lemma}

\begin{proof}
The density of $R$, with respect to its strict upper-triangular entries, is (see e.g.~\citep[Section~5.1]{muirhead1982aspects})
\[
\frac{\Gamma(d/2)^n}{\Gamma_n(d/2)}\det(R)^{(d-n-1)/2}\1\{R\succ0\}.
\]
By making the change of variables $R=I_n+d^{-1/2}Y$, we obtain the Jacobian $d^{-\binom n2/2}$. Divide the resulting density by the density of $Q_n$, which is
\[
\prod_{i<j}\phi(Y_{ij})
=(2\pi)^{-\binom n2/2}\exp(-\tr(Y^2)/4).
\]
The normalizing factor is $\Gamma(d/2)^n\Gamma_n(d/2)^{-1}(2\pi/d)^{\binom n2/2}$, whose logarithm is $K_{n,d}$.
\end{proof}

We next expand the terms in \eqref{eq:exact-cont-likelihood-new} in an effort to isolate the dominant statistics. For this proof, define
\begin{equation}\label{eq:RGG_ER_highProbEvent}
\mathcal E_n=\left\{\norm{Y}_{\op}\le C_0\sqrt n\right\},
\end{equation}
where $C_0>0$ is a sufficiently large constant to be chosen.

\begin{lemma}
\label{lem:RGG_ER_parentLLR_quadCubicQuarticFOrm}
Suppose \eqref{eq:RGG_critical_scaling} holds. Then, on $\cE_n$, for all
sufficiently large $n$,
\begin{align}
\log\frac{\dd P_{n,d}}{\dd Q_n}(Y)
={}&K_{n,d}
+\frac{n+1}{4d}\tr(Y^2)
+\frac{d-n-1}{6d^{3/2}}\tr(Y^3)
-\frac{d-n-1}{8d^2}\tr(Y^4)
+r_{n,d},
\label{eq:fourth-order-likelihood-new}
\end{align}
where the remainder satisfies
\begin{equation}
|r_{n,d}|
\leq C\frac{n\norm{Y}_{\op}^5}{d^{3/2}}=O(n^{-1}).
\label{eq:parent-likelihood-remainder}
\end{equation}
\end{lemma}

\begin{proof}
On $\cE_n$,
$
d^{-1/2}\norm{Y}_{\op}\leq C_0\sqrt{n/d}=o(1).
$
In particular, $I_n+d^{-1/2}Y\succ0$ for all sufficiently large $n$, so
the indicator term in \eqref{eq:exact-cont-likelihood-new} vanishes. Since
$\tr Y=0$, Taylor expansion of the determinant term in \eqref{eq:exact-cont-likelihood-new} gives
\begin{align*}
\frac{d-n-1}{2}\log\det(I_n+d^{-1/2}Y)
={}&-\frac{d-n-1}{4d}\tr(Y^2)
+\frac{d-n-1}{6d^{3/2}}\tr(Y^3)\\
&-\frac{d-n-1}{8d^2}\tr(Y^4)
+O\!\left(\frac{n\norm{Y}_{\op}^5}{d^{3/2}}\right).
\end{align*}
The Gaussian reference contributes $\tr(Y^2)/4$, so the combined coefficient of $\tr(Y^2)$ is
\[
-\frac{d-n-1}{4d}+\frac14=\frac{n+1}{4d},
\]
which proves \eqref{eq:fourth-order-likelihood-new} and the inequality in \eqref{eq:parent-likelihood-remainder}. Finally, on $\cE_n$, the bound is
\[
O\!\left(\frac{n^{7/2}}{d^{3/2}}\right)=O(n^{-1}),
\]
which proves the assertion in \eqref{eq:parent-likelihood-remainder}.
\end{proof}

Our next lemma gives a simplification of the constant term in Lemma \ref{lem:exact-cont-likelihood}, revealing that it is exactly the constant term appearing in \eqref{eq:continuous-lan}. 
\begin{lemma}\label{lem:normalizer}
Suppose \eqref{eq:RGG_critical_scaling} holds. Then 
\begin{equation*}
K_{n,d}=-\frac{1}{12\lambda}
+o(1).
\end{equation*}
\end{lemma}

\begin{proof}
We first show that whenever \(n=o(d)\), it holds that
\begin{equation}\label{eq:normalizer-expansion-new}
K_{n,d}=-\frac{n(n-1)(2n+5)}{24d}
+O\!\left(\frac{n^4}{d^2}\right).
\end{equation}
The result then follows immediately from the critical scaling \eqref{eq:RGG_critical_scaling}.

Put \(x=d/2\), \(m=\binom n2\), and
\(F_x(a)=\log\Gamma(x)-\log\Gamma(x-a)\).  We have the identity
\[
K_{n,d}=\sum_{r=1}^{n-1}F_x\!\inparen{\frac{r}{2}}-\frac m2\log x.
\]
Set \(\psi=\Gamma'/\Gamma\). We claim that for all $0 \leq a \leq n/2$, 
\[
F_x(a)=a\psi(x)-\frac{a^2}{2}\psi'(x)
+O\!\inparen{\frac{a^3}{x^2}},
\]
where the constant in the remainder term is independent of $n$ and $d$. 
Indeed, a Taylor expansion to second order has remainder term bounded in magnitude by $\frac{a^3}{6} \sup_{u\in[x-a,x]}|\psi''(u)|$. Then $x - a \asymp x$ and $\psi''(u) = O(u^{-2})$ gives \(\sup_{u\in[x-a,x]}|\psi''(u)|=O(x^{-2})\).

Substituting this in, we obtain
\begin{align*}
K_{n,d}
={}&\frac m2(\psi(x)-\log x)
-\frac18\left(\sum_{r=1}^{n-1}r^2\right)\psi'(x)
+O\left(\frac{\sum_{r=1}^{n-1}r^3}{x^2}\right),
\end{align*}
where we used $\sum_{r=1}^{n-1}r=m$. The first two terms are approximated by the first term on the right in \eqref{eq:normalizer-expansion-new}, using the facts that
\[
\psi(x)=\log x-\frac1{2x}+O(x^{-2}),\qquad\text{and}\qquad
\psi'(x)=\frac1x+O(x^{-2}),
\]
and
\[
2m+\sum_{r=1}^{n-1}r^2=\frac{n(n-1)(2n+5)}6.
\]
Here, the approximation error $O(n^3/d^2)$ is absorbed into the second term in \eqref{eq:normalizer-expansion-new}. Finally, for the third term, we use the fact that $\sum_{r=1}^{n-1}r^3=O(n^4)$.
\end{proof}

\subsection{Estimates for statistics of the parent matrix}

Based on Lemma \ref{lem:RGG_ER_parentLLR_quadCubicQuarticFOrm}, we focus on estimating those leading order statistics.

\begin{lemma}\label{lem:parent-moment-estimates}
Suppose \eqref{eq:RGG_critical_scaling} holds. Then under
$Q_n$, the following hold:
\begin{align}
&\tr(Y^2)=n(n-1)+O_{Q_n}(n),
\label{eq:Y-square}\\
&\tr(Y^4)
=2n(n-1)(n-2)+3n(n-1)
+o_{Q_n}(n^3).
\label{eq:Y-four-asymptotic}
\end{align}
\end{lemma}

\begin{proof}
The standard Gaussian moments show that for every fixed $r\in(0,\infty)$,
\begin{equation}\label{eq:RGG_ER_Ye_uniformMomentBound}
\E_{Q_n}|Y_{12}|^r<\infty.
\end{equation}
By independence, we have
\[
\E_{Q_n}\tr(Y^2)=n(n-1),
\quad\text{and}\quad
\Var_{Q_n}\inparen{\tr(Y^2)}
=4\binom n2\Var_{Q_n}(Y_{12}^2)=O(n^2),
\]
where we used the moment bound \eqref{eq:RGG_ER_Ye_uniformMomentBound} with $r=4$ to control the variance of $Y_{12}^2$. Statement \eqref{eq:Y-square} then follows from Chebyshev's inequality.

For the quartic trace, we use the standard trace method expansion for Wigner
matrices (see e.g.~\citep[Section~2.1]{agz2010introduction}). We have
\begin{align*}
\tr(Y^4)
={}&2\sum_eY_e^4
+4\sum_i\sum_{\substack{j<k\\j,k\ne i}}Y_{ij}^2Y_{ik}^2
+8\sum_{\{i,j,k,\ell\}}\ \sum_{\mathcal C}
\prod_{e\in\mathcal C}Y_e,
\end{align*}
where the inner sum is over the three undirected four cycles on the given
four vertices. We refer to a generic summand $Y_{ij}^2Y_{ik}^2$ in the middle sum as a
\emph{wedge monomial}.

We first compute the expectation. Since $Y_{12}$ is standard Gaussian under $Q_n$,
\[
\E_{Q_n}Y_{12}^4=3.
\]
The wedge monomials have mean one, while the four-cycle monomials have mean
zero. Thus 
\[
\E_{Q_n}\tr(Y^4)
=3n(n-1)+2n(n-1)(n-2).
\]

It remains to control the fluctuation. Let the three terms in the expansion of
$\tr(Y^4)$ be $(A)$, $(B)$, and $(C)$, from left to right. By
Cauchy--Schwarz and the uniform moment bound in
\eqref{eq:RGG_ER_Ye_uniformMomentBound} with $r=8$, the variance of each
individual monomial and the covariance of any two monomials are $O(1)$. We
first bound the three variances. Since $(A)$ is a sum of $O(n^2)$ independent terms,
$\Var_{Q_n}(A)=O(n^2)$. For $(B)$, two wedge monomials can
have nonzero covariance only if they share an edge. There are $O(n^4)$ such
ordered pairs, so $\Var_{Q_n}(B)=O(n^4)$. For $(C)$, two
four-cycle monomials can have nonzero covariance only if their edge sets
coincide. Hence $\Var_{Q_n}(C)=O(n^4)$.

We next bound the cross covariances. A summand of $(A)$ and a wedge monomial
in $(B)$ can have nonzero covariance only if they share an edge. There are
$O(n^3)$ such pairs, so
$\Cov_{Q_n}(A,B)=O(n^3)$. On the other hand,
$\Cov_{Q_n}(A,C)=0$ and
$\Cov_{Q_n}(B,C)=0$, because each corresponding product
contains an edge appearing exactly once. Therefore,
\[
\Var_{Q_n}\inparen{\tr(Y^4)}=O(n^4).
\]
Consequently, Chebyshev's inequality gives
$\tr(Y^4)-\E_{Q_n}\tr(Y^4)
=o_{Q_n}(n^3)$, proving
\eqref{eq:Y-four-asymptotic}.
\end{proof}

\subsection{Limit of log-likelihood ratios}

\begin{proof}[Proof of Theorem \ref{thm:continuous-lan}]
Under $Q_n$, the entries $(Y_{ij})_{i<j}$ are independent standard Gaussians and $Y_{ii}=0$. By \citet[Corollary~4.4.8]{vershynin2018high}, we have $Q_n(\mathcal E_n^c)\leq4e^{-n}\to0$ for sufficiently large fixed $C_0$. On $\cE_n$, Lemma~\ref{lem:RGG_ER_parentLLR_quadCubicQuarticFOrm} gives \eqref{eq:fourth-order-likelihood-new} for sufficiently large $n$.

We first show that the cubic term in \eqref{eq:fourth-order-likelihood-new} is indeed the asymptotically dominant statistic that appears in Theorem \ref{thm:continuous-lan}.  
First,
\[
S_n(Y)=O_{Q_n}(1).
\]
To see this, find by direct calculation that
\begin{align*}
\E_{Q_n}S_n(Y)&=0,\\
\Var_{Q_n}\inparen{S_n(Y)}
&=\EE[Y_{12}^2 Y_{13}^2 Y_{23}^2]=1.
\end{align*}
The claim then follows from
Chebyshev's inequality.

Note that
\(\tr(Y^3)=6\sum_{i<j<k}Y_{ij}Y_{ik}Y_{jk} = 6\sqrt{\binom n3}\,S_n(Y)\). Since $\binom n3/d\to1/(6\lambda)$, the cubic term in
\eqref{eq:fourth-order-likelihood-new} becomes 
\begin{equation*}
\frac{d - n - 1}{6 d^{3/2}} \tr(Y^3)
=\inparen{1 - \frac{n+1}{d}}\sqrt{\frac{\binom n3}{d}}\,S_n(Y)
=\frac{S_n(Y)}{\sqrt{6\lambda}}+o_{Q_n}(1).
\end{equation*}

We now apply Lemma \ref{lem:parent-moment-estimates} to determine the combined contribution of the remaining terms in \eqref{eq:fourth-order-likelihood-new}. By \eqref{eq:Y-square},
\[
\frac{n+1}{4d}\tr(Y^2)
=\frac{n(n-1)(n+1)}{4d}
+o_{Q_n}(1).
\]
Similarly,
\eqref{eq:Y-four-asymptotic} shows that
\[
-\frac{d-n-1}{8d^2}\tr(Y^4) = -\frac{n(n-1)(n-2)}{4d}
+o_{Q_n}(1).
\]
The difference between the two deterministic terms is $3n(n-1)/(4d)=o(1)$. Hence the quadratic and quartic terms cancel up to a vanishing remainder. Finally,
combining with Lemma~\ref{lem:normalizer}, we obtain a remainder $\epsilon_n=o_{Q_n}(1)$ such that, on $\cE_n$ for all sufficiently large $n$,
\[
\log\frac{\dd P_{n,d}}{\dd Q_n}(Y) =\frac{S_n(Y)}{\sqrt{6\lambda}}-\frac{1}{12\lambda}
+\epsilon_n.
\]
Since $Q_n(\cE_n^c)\to0$, this implies the unconditional expansion \eqref{eq:continuous-lan} and finishes the proof.
\end{proof}

\begin{lemma}\label{lem:cubic-statistic-clt}
Under $Q_n$,
\begin{equation*}
S_n(Y)\convd\cN(0,1).
\end{equation*}
\end{lemma}

\begin{proof}
The central limit theorem for the Wigner moments \citep[Theorem~2.1.31 and (2.1.44)]{agz2010introduction} applied to $n^{-1/2}Y$ with power $3$ gives
\[
\frac{\tr(Y^3)}{n^{3/2}}\convd\cN(0,6).
\]
Since $\tr(Y^3)=6\sqrt{\binom n3}\,S_n(Y)$ and $6\sqrt{\binom n3}/n^{3/2}\to\sqrt6$, the conclusion follows.
\end{proof}

\begin{proof}[Proof of Proposition~\ref{prop:parent-asymptotics}]
Set \(I=1/(6\lambda)\) and
\(L_{n,d}=\dd P_{n,d}/\dd Q_n\).
Theorem~\ref{thm:continuous-lan} and Lemma~\ref{lem:cubic-statistic-clt}
give \(\log L_{n,d}\convd\cN(-I/2,I)\) under \(Q_n\).
As in the proof of Theorem~\ref{thm:RGG_ER_LLR_asymptotics}, Le Cam's third lemma
(Lemma~\ref{lem:le-cam-third}) gives \(\log L_{n,d}\convd\cN(I/2,I)\) under
\(P_{n,d}\), proving \eqref{eq:parent-log-limits}.
The TV limit \eqref{eq:parent-tv} follows from the Neyman--Pearson lemma by
thresholding \(\log L_{n,d}\) at zero and using these two Gaussian limits.
\end{proof}

\section{Additional proofs for random graphs}\label{app:rgg}

\subsection{Distance between random geometric graphs}

Recall that the auxiliary distribution $\widetilde P_{n,d}^{\mathsf{G}}$ defined in Section~\ref{sec:assembly} uses the Gaussian threshold $z_p$ and hence does not have edge density exactly $p$. We now show that it is sufficiently close to the distribution $P_{n,d}^{\mathsf{G}}$ with edge density $p$.

For $x_1,x_2 \overset{\text{iid}}{\sim} \Unif(\SSS^{d-1})$, the density function of $\sqrt d\langle x_1,x_2\rangle$ is (see e.g.~Theorem 1.5.7(i) of \citet{muirhead1982aspects})
\begin{equation*}
h_d(u)=\frac{c_d}{\sqrt d}
\left(1-\frac{u^2}{d}\right)^{(d-3)/2}
\1\{|u|<\sqrt d\},\qquad\text{where}\qquad
c_d=\frac{\Gamma(d/2)}{\sqrt\pi\,\Gamma((d-1)/2)}.
\end{equation*}

\begin{lemma}\label{lem:rgg-threshold-correction}
For fixed \(p\in(0,1)\),
\begin{equation*}
\TV(\widetilde P_{n,d}^{\mathsf{G}},P_{n,d}^{\mathsf{G}})=O(n^2/d),
\end{equation*}
where the implicit constant may depend on $p$.
\end{lemma}

\begin{proof}
Since $p$ is fixed, Stirling's formula and the Taylor expansion
give, uniformly for $|u|\le |z_p|$ and all sufficiently large $d$,
\[
\frac{c_d}{\sqrt d}=\phi(0)\left(1+O(d^{-1})\right),
\qquad
\frac{d-3}{2}\log\left(1-\frac{u^2}{d}\right)
=-\frac{u^2}{2}+O(d^{-1}),
\]
so the density formula implies $h_d(u)=\phi(u)(1+O(d^{-1}))$ uniformly
on $[-|z_p|,|z_p|]$. In particular, $\sup_{|u|\le |z_p|}|h_d(u)-\phi(u)|=O(d^{-1})$.

By symmetry, $H_d(0)=\Phi(0)=1/2$ where $H_d$ is the CDF of $h_d$ and $\Phi$ is the standard Gaussian CDF. As a result,
\[
|H_d(z_p)-\Phi(z_p)|
=\left|\int_0^{z_p}(h_d(u)-\phi(u))\dd u\right|
\le |z_p|\sup_{|u|\le |z_p|}|h_d(u)-\phi(u)|
=O(d^{-1}).
\]

Now couple $A\sim P_{n,d}^{\mathsf{G}}$ and
$\widetilde A\sim\widetilde P_{n,d}^{\mathsf{G}}$ using the same latent points. An edge differs exactly
when $Y_{ij}=\sqrt d\langle x_i,x_j\rangle$ lies between the thresholds $\tau_{d,p}$ and $z_p$.
Since $H_d(\tau_{d,p})=1-p=\Phi(z_p)$, 
\[
\Pp(A_{ij}\ne\widetilde A_{ij})
=|H_d(\tau_{d,p})-H_d(z_p)|
=|\Phi(z_p)-H_d(z_p)|=O(d^{-1}).
\]
Consequently, 
\begin{align*}
\TV(\widetilde P_{n,d}^{\mathsf{G}},P_{n,d}^{\mathsf{G}})
&=\sup_B\left|\E\left[\1\{A\in B\}-\1\{\widetilde A\in B\}\right]\right|\\
&\le\sup_B\E\left|\1\{A\in B\}-\1\{\widetilde A\in B\}\right|\\
&\le \Pp(A\ne\widetilde A)
=O(n^2/d).
\end{align*}
\end{proof}

\subsection{The conditional limit}
\label{sec:residual}

Before proving Proposition~\ref{prop:quenched-residual}, we first show a joint central limit theorem for cubic statistics.

\begin{lemma}\label{lem:triangle-vector-clt}
Fix an integer \(m\ge1\). For each \(n\), let
\(Z_{ij}^{(n)}=(Z_{ij,1}^{(n)},\ldots,Z_{ij,m}^{(n)})\in\R^m\),
\(1\le i<j\le n\), be i.i.d.\ centered
random vectors. Assume that
\(\sup_{n\ge2}\E\|Z_{12}^{(n)}\|^k<\infty\) for every fixed integer \(k\ge1\),
and that for $r,s\in[m]$, as $n \to \infty$, 
\[
\E[Z_{12,r}^{(n)}Z_{12,s}^{(n)}]\longrightarrow\Sigma_{rs}.
\]
For each coordinate $r\in[m]$, define
\[
W_r^{(n)}=\frac1{\sqrt{\binom{n}{3}}}
\sum_{i<j<k}Z_{ij,r}^{(n)}Z_{ik,r}^{(n)}Z_{jk,r}^{(n)}.
\]
Then \(W^{(n)}:=(W_1^{(n)},\ldots,W_m^{(n)})\) converges in distribution to a centered
Gaussian vector with covariance matrix
\(\big(\Sigma_{rs}^3\big)_{r,s\in[m]}\).
\end{lemma}

\begin{proof}
We prove convergence of moments. 
Fix \(k\ge1\) and \(r_1,\ldots,r_k\in[m]\). Expansion gives
\begin{equation}
\E\prod_{\ell=1}^kW_{r_\ell}^{(n)}
=\binom n3^{-k/2}
\sum_{T_1,\ldots,T_k\in\binom{[n]}3}M_n(\mathbf T), \label{eq:moment-k-expansion}
\end{equation}
where \(\mathbf T=(T_1,\ldots,T_k)\) is a list of triangles and
$$
M_n(\mathbf T)
:=\E\left[\prod_{\ell=1}^k
\prod_{\{i,j\}\in\binom{T_\ell}2}Z_{ij,r_\ell}^{(n)}\right].
$$

Let \(H\) be the graph union of $T_1,\ldots,T_k$, with
\(v=|V(H)|\) and \(e=|E(H)|\). Independence across edges gives
\[
M_n(\mathbf T)
=\prod_{\{i,j\}\in E(H)}
\E\left[\prod_{\ell:\,\{i,j\}\subset T_\ell}
Z_{ij,r_\ell}^{(n)}\right].
\]
Since each coordinate is centered,
\[
\#\{\ell:\{i,j\}\subset T_\ell\}=1
\quad\Longrightarrow\quad M_n(\mathbf T)=0.
\]
Thus a nonzero summand has every edge occurring at least twice.
Every vertex in $H$ also has degree at least two, so
\begin{equation}
2v\le\sum_{u\in V(H)}\deg_H(u)=2e\le3k. \label{eq:vertex-number-bound}
\end{equation}

H\"older's inequality bounds each summand in
\eqref{eq:moment-k-expansion} uniformly in \(n\) and \(\mathbf T\):
\begin{align*}
|M_n(\mathbf T)|
&\le\E\prod_{\ell=1}^k
\prod_{\{i,j\}\in\binom{T_\ell}2}|Z_{ij,r_\ell}^{(n)}|\\
&\le\prod_{\ell=1}^k\prod_{\{i,j\}\in\binom{T_\ell}2}
\left(\E|Z_{ij,r_\ell}^{(n)}|^{3k}\right)^{1/(3k)}\\
&\le\prod_{\ell=1}^k\prod_{\{i,j\}\in\binom{T_\ell}2}
\left(\E\|Z_{ij}^{(n)}\|_2^{3k}\right)^{1/(3k)}\\
&=\E\left[\|Z_{12}^{(n)}\|_2^{3k}\right]
\le\sup_{n\ge2}\E\left[\|Z_{12}^{(n)}\|_2^{3k}\right]<\infty.
\end{align*}
Next, the number of triangle lists whose union has $v$ vertices satisfies
\[
\#\{\mathbf T:|V(H)|=v\}
\le\binom nv\binom v3^k=O(n^v).
\]
Combining the above bounds with \(\binom n3^{k/2}\asymp n^{3k/2}\) gives
\begin{equation}
\binom n3^{-k/2}
\sum_{\mathbf T:\,|V(H)|<3k/2}|M_n(\mathbf T)|
= O \bigg(\sum_{3\le v<3k/2}n^{v-3k/2}\bigg) = o(1) . \label{eq:MnT-sum-vanishing-part}
\end{equation}

It follows from \eqref{eq:moment-k-expansion}, \eqref{eq:vertex-number-bound}, and \eqref{eq:MnT-sum-vanishing-part} that 
\[
\E\prod_{\ell=1}^kW_{r_\ell}^{(n)}
=\binom n3^{-k/2}
\sum_{\substack{T_1,\ldots,T_k\in\binom{[n]}3\\|V(H)|=3k/2}}
M_n(\mathbf T)+o(1).
\]
If \(k\) is odd, the above sum is empty and the mixed moment tends to zero. 

Now suppose \(k\) is even.
By \eqref{eq:vertex-number-bound}, every nonzero summand above
satisfies $2v=2e=3k$. Consequently, $H$ is a disjoint union
of triangles, each appearing exactly twice in $\mathbf T$.
Let \(\mathcal P_2(k)\) denote the pairings
of \([k]\). Fix a pairing \(\pi\in\mathcal P_2(k)\) of triangles in $\mathbf T$.  
Counting the set of paired triangles gives
\[
\#\{\mathbf T\text{ corresponding to }\pi\}
=\prod_{q=0}^{k/2-1}\binom{n-3q}{3}
=\frac{(n)_{3k/2}}{6^{k/2}} \sim \binom n3^{k/2} ,
\]
where \((n)_j=n(n-1)\cdots(n-j+1)\). 
For the triangle pairing $\pi$, independence across edges gives
\[
M_n(\mathbf T)
=\prod_{\{\ell,\ell'\}\in\pi}
\left(\E[Z_{12,r_\ell}^{(n)}Z_{12,r_{\ell'}}^{(n)}]\right)^3.
\]
It then follows from the assumption on the covariance that
\[
M_n(\mathbf T)\longrightarrow
\prod_{\{\ell,\ell'\}\in\pi}\Sigma_{r_\ell r_{\ell'}}^3.
\]
Combining the results, we conclude that
\begin{equation}
\E\prod_{\ell=1}^kW_{r_\ell}^{(n)}
\longrightarrow
\sum_{\pi\in\mathcal P_2(k)}
\prod_{\{\ell,\ell'\}\in\pi}\Sigma_{r_\ell r_{\ell'}}^3. \label{eq:W-limiting-moments}
\end{equation}
Note that the above limit also holds for odd \(k\) as \(\mathcal P_2(k)=\varnothing\).

To identify the limiting distribution from the asymptotic moments in \eqref{eq:W-limiting-moments}, we first take \(k=2\) to obtain 
\[
\operatorname{Cov}(W_r^{(n)},W_s^{(n)})
=\left(\E[Z_{12,r}^{(n)}Z_{12,s}^{(n)}]\right)^3.
\]
Thus \((\Sigma_{rs}^3)_{r,s\in[m]}\) is a limit of covariance matrices and
is positive semidefinite. Let
\[
G\sim\cN(0,(\Sigma_{rs}^3)_{r,s\in[m]}).
\]
Since \(\E[G_rG_s]=\Sigma_{rs}^3\), Wick's formula gives
\[
\E\prod_{\ell=1}^kG_{r_\ell}
=\sum_{\pi\in\mathcal P_2(k)}
\prod_{\{\ell,\ell'\}\in\pi}\Sigma_{r_\ell r_{\ell'}}^3,
\]
which is exactly the asymptotic mixed moment in \eqref{eq:W-limiting-moments}.
This proves asymptotic normality.
\end{proof}

\begin{proof}[Proof of Proposition~\ref{prop:quenched-residual}]
For any $(i,j)$, recall that $Y_{ij}\sim\cN(0,1)$ under $Q_n$, $A_{ij}=\1\{Y_{ij}\ge z_p\}$, and $\bar A_{ij}=\frac{A_{ij}-p}{\sqrt{p(1-p)}}$. 
It is straightforward to obtain $\E_{Q_n}[Y_{ij}\mid A_{ij}]=\rho_p\bar A_{ij}$ where $\rho_p$ is given in \eqref{eq:rho-p-correlation}.
As a result,
\begin{equation*}
\E_{Q_n}\insquare{S_n(Y)\mid A}
=\frac{\rho_p^3}{\sqrt{\binom n3}}
\sum_{i<j<k}\bar A_{ij}\bar A_{ik}\bar A_{jk}
=\rho_p^3 S_n^{\mathsf{G}}(A).
\end{equation*}

Draw $A\sim Q_n^{\mathsf{G}}$ and, conditionally on $A$, draw
$Y^{(1)}$ and $Y^{(2)}$ independently with common distribution
$\cL_{Q_n}(Y\mid A)$. Define the two residuals by
\begin{equation}\label{eq:rgg-replica-residuals}
R_n^{(\ell)}
=S_n(Y^{(\ell)})-\rho_p^3 S_n^{\mathsf{G}}(A),
\qquad \ell=1,2.
\end{equation}
By Lemma~\ref{lem:two-replica-criterion},
it suffices to show that they converge jointly to independent
$\cN(0,1-\rho_p^6)$ variables.

Note that
\[
(\bar A_{ij},Y_{ij}^{(1)},Y_{ij}^{(2)}),\qquad 1\le i<j\le n,
\]
are i.i.d.\ across edges, and each entry is standardized marginally. 
Moreover,
\begin{equation*}
\E[\bar A_{ij}Y_{ij}^{(\ell)}]=\rho_p,\qquad
\E[Y_{ij}^{(1)}Y_{ij}^{(2)}]
=\E\left[\E[Y_{ij}^{(1)}\mid A]\E[Y_{ij}^{(2)}\mid A]\right]
=\rho_p^2 .
\end{equation*}
Thus the covariance matrix of $(\bar A_{ij},Y_{ij}^{(1)},Y_{ij}^{(2)})$ is
\begin{equation*}
\Sigma=
\begin{pmatrix}
1&\rho_p&\rho_p\\
\rho_p&1&\rho_p^2\\
\rho_p&\rho_p^2&1
\end{pmatrix}.
\end{equation*}
Applying Lemma~\ref{lem:triangle-vector-clt} with $m=3$ gives
\begin{equation*}
\left(S_n^{\mathsf{G}}(A),
S_n(Y^{(1)}),
S_n(Y^{(2)})\right)
\convd
\cN\!\left(0,
\begin{pmatrix}
1&\rho_p^3&\rho_p^3\\
\rho_p^3&1&\rho_p^6\\
\rho_p^3&\rho_p^6&1
\end{pmatrix}\right).
\end{equation*}
This, together with 
\eqref{eq:rgg-replica-residuals}, gives that 
\[
(R_n^{(1)},R_n^{(2)})
\convd
\cN(0,1-\rho_p^6)\otimes\cN(0,1-\rho_p^6).
\]
Applying Lemma~\ref{lem:two-replica-criterion} with $\cG_n=\sigma(A)$ then completes the proof.
\end{proof}

\section{Additional proofs for rank statistics}\label{app:ranks}

\subsection{Approximating the cubic statistic with ranks}
\label{subsec:rank-reconstruction}

This section proves Lemma~\ref{lem:rank-statistic-reconstruction}.
Throughout, all expectations \(\E\) are taken under \(Q_n\).
Let \(m=\binom n2\), and label the \(m\) edge
coordinates by \([m]\). For \(r\in[m]\), define
\[
G_r=g_{\Pi(r)},
\]
which is the Gaussian quantile corresponding to \(Y_r\). For any
\(I\in\binom{[m]}{3}\), write
\[
Y_I=\prod_{r\in I}Y_r,\qquad\text{and}\qquad G_I=\prod_{r\in I}G_r.
\]
(The three edges in $I$ may not form a triangle.) 
The pairs
\((Y_r,G_r)_{r\in[m]}\) are exchangeable under all permutations of the
edge coordinates. We first prove two lemmas.

\begin{lemma}\label{lem:rank-alt-preliminary}
We have
\begin{equation}
\max_{r\in[m]}\bigl(\E|Y_r|^6+\E|G_r|^6\bigr) = O(1). \label{eq:six-moment-bound}
\end{equation}
Moreover, for every \(I\in\binom{[m]}3\),
\begin{equation}
\E(Y_I-G_I)^2\longrightarrow0.  \label{eq:cubic-L2-conv}
\end{equation}
\end{lemma}

\begin{proof}
The random variable \(\Phi(Y_1)\) is uniform on \((0,1)\), and \(\Pi(1)\) is its rank
among the \(m\) i.i.d.\ uniform random variables \((\Phi(Y_r))_{r\in[m]}\).
Define \(h(x)=x^3\).
The function \(h\circ\Phi^{-1}\) is increasing and square-integrable. Thus
Example~13.9 of \citet{vanderVaart1998asymptotic} gives
\[
\E\left|h(Y_1)-h(G_1)\right|^2\longrightarrow0.
\]
For all \(a,b\in\R\), $|a-b|^6\le16|h(a)-h(b)|^2$, 
so
\begin{equation}\label{eq:rank-alt-coordinate-error}
\E|Y_1-G_1|^6
\longrightarrow0.
\end{equation}
The triangle inequality and exchangeability then give \eqref{eq:six-moment-bound}.

For \eqref{eq:cubic-L2-conv}, by exchangeability, it suffices to consider \(I=\{1,2,3\}\). We have 
\begin{align*}
(Y_I-G_I)^2
&=\left( (Y_1-G_1)Y_2Y_3+G_1(Y_2-G_2)Y_3+G_1G_2(Y_3-G_3) \right)^2\\
&\le3\bigl[(Y_1-G_1)^2Y_2^2Y_3^2+G_1^2(Y_2-G_2)^2Y_3^2
+G_1^2G_2^2(Y_3-G_3)^2\bigr].
\end{align*}
Applying H\"older's inequality to the first term
and using \eqref{eq:six-moment-bound} and \eqref{eq:rank-alt-coordinate-error} gives
\[
\E\bigl[(Y_1-G_1)^2Y_2^2Y_3^2\bigr]
\le\bigl(\E|Y_1-G_1|^6\,\E|Y_2|^6\,\E|Y_3|^6\bigr)^{1/3}
\le C\bigl(\E|Y_1-G_1|^6\bigr)^{1/3}\longrightarrow0.
\]
The other two terms are bounded in the same way, proving \eqref{eq:cubic-L2-conv}.
\end{proof}

\begin{lemma}\label{lem:rank-alt-mixed-moments}
There is an absolute constant \(C\) such that, for
\(I,J\in\binom{[m]}3\),
\[
|\E[Y_IG_J]|+|\E[G_IG_J]|
\le C
\begin{cases}
m^{-1},& \text{ if } |I\cap J|=1,\\
m^{-2},& \text{ if } |I\cap J|=0.
\end{cases}
\]
\end{lemma}

\begin{proof}
Define 
$$
p_k:=\sum_{r=1}^mG_r^k=\sum_{r=1}^mg_r^k.
$$ 
By Gaussian symmetry, we have \(g_{m+1-r}=-g_r\), so
$$
p_1=p_3=0 .
$$
Note that $p_2$ is deterministic, so by \eqref{eq:six-moment-bound},
$$
p_2 = \E[p_2] = \E\bigg[\sum_{r=1}^mG_r^2\bigg]
=m\E G_1^2\le m(\E|G_1|^6)^{1/3} = O(m).
$$

Next, using \(p_1=p_3=0\), we obtain
\begin{align}
\sum_{J\in\binom{[m]}3}G_J
&=\frac16\sum_{r=1}^mG_r
\bigg[\bigg(\sum_{s\ne r}G_s\bigg)^2-\sum_{s\ne r}G_s^2\bigg]\notag\\
&=\frac16\sum_{r=1}^mG_r
\bigl[(p_1-G_r)^2-(p_2-G_r^2)\bigr]\notag\\
&=\frac16\sum_{r=1}^m (2G_r^3-p_2 G_r)
=\frac{2p_3-p_1p_2}{6}=0,
\label{eq:rank-alt-score-total}
\end{align}
and
\begin{align}
\sum_{J\in\binom{[m]}3, \, 1\in J}G_J
&=\frac{G_1}{2}
\bigg[\bigg(\sum_{r=2}^mG_r\bigg)^2-\sum_{r=2}^mG_r^2\bigg]\notag\\
&=\frac{G_1}{2}\bigl[(p_1-G_1)^2-(p_2-G_1^2)\bigr]
=G_1^3-\frac{p_2}{2}G_1.
\label{eq:rank-alt-score-containing-one}
\end{align}

We now bound \(\E[Y_IG_J]\). By exchangeability, we may fix
\(I=\{1,2,3\}\) and define
\[
a_k:=\E[Y_IG_J],\qquad \text{for }\,  |I\cap J|=k\in\{0,1,2,3\}.
\]
H\"older's inequality and \eqref{eq:six-moment-bound} give
\[
|a_k|\le\prod_{r\in I}\bigl(\E|Y_r|^6\bigr)^{1/6}
\prod_{s\in J}\bigl(\E|G_s|^6\bigr)^{1/6}=O(1).
\]
Similarly, \(|\E[Y_IG_1^3]|+|\E[Y_IG_1]|=O(1)\).
Multiplying \eqref{eq:rank-alt-score-containing-one}
by \(Y_I\), taking expectations, and using \(p_2=O(m)\), we obtain
\[
\sum_{J\in\binom{[m]}3,\,1\in J}\E[Y_IG_J]
=\E[Y_IG_1^3]-\frac{p_2}{2}\E[Y_IG_1]
=O(m).
\]
Since \(I=\{1,2,3\}\), each triple containing \(1\) is of the form
\(\{1,r,s\}\), \(\{1,2,r\}\), \(\{1,3,r\}\), or \(I\),
where \(r,s \in \{4,\ldots,m\}\). Thus, counting the number of choices of $(r,s)$ in each case gives
\begin{equation*}
\sum_{J\in\binom{[m]}3,\,1\in J}\E[Y_IG_J]
=\binom{m-3}{2}a_1+2(m-3)a_2+a_3=O(m).
\end{equation*}
Since \(a_2,a_3=O(1)\), it follows that
\(a_1=O(m^{-1})\). Similarly, \eqref{eq:rank-alt-score-total} gives
\begin{equation*}
\binom{m-3}{3}a_0+3\binom{m-3}{2}a_1+3(m-3)a_2+a_3=0.
\end{equation*}
Using \(a_1=O(m^{-1})\) and \(a_2,a_3=O(1)\) then yields 
\(a_0=O(m^{-2})\). 

The same argument with \(Y_I\) replaced by
\(G_I\) gives the bound on \(\E[G_IG_J]\).
\end{proof}

\begin{proof}[Proof of Lemma~\ref{lem:rank-statistic-reconstruction}]
Let \(\mathcal T_n\subset\binom{[m]}3\) be the collection of edge sets
that form triangles, so \(|\mathcal T_n|=\binom n3\). Expanding the square gives
\begin{equation}\label{eq:rank-alt-error-expansion}
\E\bigl(S_n(Y)-S_n^{\mathsf{R}}(\Pi)\bigr)^2
=\frac1{\binom n3}\sum_{I,J\in\mathcal T_n}\E[(Y_I-G_I)(Y_J-G_J)].
\end{equation}
For \(I\ne J\), independence and centering of the \(Y_r\)'s give
\(\E[Y_IY_J]=0\). Expanding the product and applying
Lemma~\ref{lem:rank-alt-mixed-moments} to both \((I,J)\) and \((J,I)\),
we therefore obtain
\[
|\E[(Y_I-G_I)(Y_J-G_J)]|
\le C
\begin{cases}
m^{-1},&|I\cap J|=1,\\
m^{-2},&|I\cap J|=0.
\end{cases}
\]
For each fixed triangle \(I\), there are \(O(n)\) triangles sharing
exactly one edge with $I$ and \(O(n^3)\) triangles with no common edge.
The diagonal terms in \eqref{eq:rank-alt-error-expansion} are all equal
by exchangeability. Hence
\begin{equation*}
\E\bigl(S_n(Y)-S_n^{\mathsf{R}}(\Pi)\bigr)^2
=\E\bigl(Y_1Y_2Y_3-G_1G_2G_3\bigr)^2
+O\bigg(\frac{n}{m}+\frac{\binom n3}{m^2}\bigg) \longrightarrow0,
\end{equation*}
where we used \eqref{eq:cubic-L2-conv} and
\(m = \binom{n}{2}\). This finishes the proof.
\end{proof}

\section*{AI Disclosure}

The original proof of the asymptotic error for testing the random geometric graph model against the Erd\H{o}s--R\'enyi model at the critical threshold $d/n^3 \to \lambda$ was found by GPT-5.6. The authors substantially redeveloped the proofs with the assistance of GPT-5.6 and GPT-6, including connecting to the classical theory of Davies and Le Cam--Yang, simplifying various technical steps, and extending the results to rank statistics and kernel matrices. All the results in Section~\ref{sec:main-results} have been formalized in Lean using GPT-6.1. The authors are fully responsible for this work.

\section*{Acknowledgments}

We are very grateful to Will Perkins for invaluable discussions about sharp thresholds. Cheng Mao was supported in part by NSF CAREER Award 2338062.

\bibliographystyle{plainnat}
\bibliography{ref}

\end{document}